\documentclass[12pt, a4paper]{amsart}
\usepackage{amscd,amssymb,mathtools,mathrsfs}
\usepackage[shortlabels]{enumitem}
\usepackage{tensor}
\usepackage{wrapfig}

\allowdisplaybreaks

\usepackage{color, url, hyperref}

\usepackage{soul}

\usepackage{graphicx}

\usepackage{adjustbox}

\usepackage{tikz}
\usetikzlibrary{cd,calc,overlay-beamer-styles,decorations.pathmorphing,arrows.meta}

\tikzset{
    vertex/.style={circle, fill=black, scale=0.3, outer sep=5pt},
    open-vertex/.style={circle, draw=black, scale=0.3, outer sep=5pt},
}

\hypersetup{colorlinks=true, linkcolor=blue, linkbordercolor=blue, citecolor=blue, citebordercolor=blue, linktocpage=true}

\DeclareMathOperator{\coker}{coker}

\newcommand{\ebar}{{\overline{e}}}

\DeclarePairedDelimiter{\abs}{\lvert}{\rvert}

\def\N{\mathbb{N}}
\def\Z{\mathbb{Z}}
\def\T{\mathbb{T}}

\def\Q{\mathbb{Q}}

\def\GG{\mathcal{G}}

\def\MM{\mathcal{M}}

\def\OO{\mathcal{O}}

\def\SS{\mathcal{S}}

\newtheorem{thm}{Theorem}
\newtheorem{lemma}[thm]{Lemma}
\newtheorem{prop}[thm]{Proposition}
\newtheorem{coro}[thm]{Corollary}

\theoremstyle{definition}
\newtheorem{defin}[thm]{Definition}

\newtheorem{egs}[thm]{Examples}

\theoremstyle{remark}
\newtheorem{remark}[thm]{Remark}
\newtheorem{notation}[thm]{Notation}
\newtheorem{construction}[thm]{Construction}

\begin{document}

\date{\today}
\title[$C^*$-algebras associated to directed graphs of groups]{$C^*$-algebras associated to directed graphs of groups, and models of Kirchberg algebras}

\author[Wu]{Victor Wu}
\address{Victor Wu, School of Mathematics and Statistics \\
    The University of Sydney}
\email{viwu8694@uni.sydney.edu.au}

\begin{abstract}
    We introduce $C^*$-algebras associated to directed graphs of groups. In particular, we associate a combinatorial $C^*$-algebra to each row-finite directed graph of groups with no sources, and show that this $C^*$-algebra is Morita equivalent to the crossed product coming from the corresponding group action on the boundary of a directed tree. Finally, we show that these $C^*$-algebras (and their Morita equivalent crossed products) contain the class of stable UCT Kirchberg algebras.
\end{abstract}

\maketitle

\section{Introduction}

\renewcommand*{\thethm}{\Alph{thm}}

Universal $C^*$-algebras for generators and relations encoding an underlying combinatorial object have long been studied because of their tractable nature, and their ability to prove rich sources of examples. The study of these $C^*$-algebras, commonly referred to as combinatorial $C^*$-algebras, began with the work of Cuntz and Krieger, who in \cite{CuntzKrieger} associated a $C^*$-algebra to a finite $\{0,1\}$-matrix. These algebras developed into directed graph $C^*$-algebras, which in turn have branched out into various generalisations, including $C^*$-algebras of higher rank graphs \cite{higher_rank_graphs} and left cancellative small categories \cite{LCSC} (see also \cite{Raeburn} for an overview).

One abundant class of combinatorial objects are \textit{graphs of groups}, which were introduced by Serre in \cite{Serre} and further studied by Bass in \cite{Bass}. Roughly speaking, these are (undirected) graphs where each vertex and edge has an associated group, and each edge group embeds into its adjacent vertex groups. The key idea of what is known as Bass--Serre theory is that graphs of groups ``encode" group actions on trees: given a group acting on a tree (with a mild assumption), we can construct a graph of groups from the quotient graph and the stabiliser subgroups at the vertices and edges; and from this graph of groups we can construct a group, called the \textit{fundamental group}, and an action of this group on a tree, called the \textit{Bass--Serre tree}, which is isomorphic to the original action. We provide a more detailed review of Bass--Serre theory in Section~\ref{sec:background}, which also contains other background definitions and theory necessary for the rest of the paper.

Graphs of groups were first studied in a $C^*$-algebraic setting in \cite{Okayasu}, which considered finite graphs of finite groups; these were then generalised to countable graphs of countable groups (with some regularity assumptions) in \cite{BMPST}. The authors of \cite{BMPST} also proved a `$C^*$-algebraic Bass--Serre theorem', which states that their graph of groups $C^*$-algebra is stably isomorphic to the crossed product coming from the action of the fundamental group on the boundary of the Bass--Serre tree.

In this paper, we consider what we call \textit{directed} graphs of groups: these are simply graphs of groups where the graph additionally carries an orientation. While these objects are not so interesting from a geometric group theory perspective (as the orientation data does not fundamentally change anything in the theory), they have found interest in other fields: for example, it was shown in \cite{levi_categories} that directed graph of groups are in a one-to-one correspondence with skeletal cancellative Levi categories.

Our motivation for studying directed graphs of groups comes from the observation that Bass--Serre theory implies that directed graphs of groups are in a one-to-one correspondence with group actions on \textit{directed} trees: we can simply `lift' the orientation of a directed graph of groups to the Bass--Serre tree. These actions fit into the framework studied in \cite{BSTW}, which provides structural theorems and $K$-theoretic formulae for crossed products coming from such actions.

In Section~\ref{sec:dgogs} we develop the theory of directed graphs of groups that we will use in the rest of the paper. As part of this, we give an alternate approach to associating a category (which we will call the \textit{word category}) to a directed graph of groups; this category matches the one from \cite{levi_categories}, but the construction is more direct and is more natural for our purposes.

Inspired by the relations in \cite{BMPST}, we associate a $C^*$-algebra to each directed graph of groups in Section~\ref{sec:dgog algebra}, where we also show that this algebra is isomorphic to the $C^*$-algebra of the word category of the directed graph of groups. These directed graph of groups $C^*$-algebras generalise directed graph $C^*$-algebras (in that for a directed graph of groups where all groups are trivial, the $C^*$-algebra of the directed graph of groups is isomorphic to that of the underlying directed graph), and many properties of directed graph $C^*$-algebras generalise to this broader class of $C^*$-algebras as well.

One such property is that these $C^*$-algebras are Morita equivalent to certain crossed products by commutative $C^*$-algebras. For a directed graph $E$, it was shown in \cite{Kumjian-Pask} that its $C^*$-algebra $C^*(E)$ is Morita equivalent to the crossed product $C_0(\partial T) \rtimes G$, where $T$ is the universal covering tree of $E$ and $G$ is the fundamental group of $E$, which acts naturally on $T$. For a directed graph of groups $\GG_+$, we have the following generalisation (this can also be viewed as a `$C^*$-algebraic \textit{directed} Bass--Serre theorem', in line with \cite{BMPST}).

\begin{thm}[Theorem~\ref{thm:morita equiv}]
    Let $\GG_+ = (\Gamma_+, G)$ be a countable, row-finite (connected) directed graph of groups with no sources, and choose a base vertex $x \in \Gamma^0$. Write $\pi_1(\GG, x)$ for the fundamental group of $\GG$ based at $x$, and write $X_{\GG_+, x}$ for the directed Bass--Serre tree based at $x$. Then $C^*(\GG_+)$ is Morita equivalent to $C_0(\partial X_{\GG_+,x}) \rtimes \pi_1(\GG,x)$.
\end{thm}

Section~\ref{sec:C* Bass-Serre thm} is dedicated to the proof of this theorem, which has two steps: first we find a tractable groupoid model for $C^*(\GG_+)$ with some help from the general results in \cite{LCSCsemigroup} concerning groupoid models for the $C^*$-algebras of categories; then we show that this groupoid is equivalent to the transformation groupoid $\pi_1(\GG, x) \ltimes \partial X_{\GG_+,x}$, whose $C^*$-algebra is the crossed product in the theorem statement.

Since Morita equivalence preserves $K$-theory and many structural properties of $C^*$-algebras such as simplicity and nuclearity, we can use the results of \cite{BSTW} to study these properties for directed graph of groups $C^*$-algebras. In Section~\ref{sec:Kirchberg}, the final section of this paper, we study these properties for the special case of directed graphs of \textit{infinite cyclic} groups (that is, directed graphs of groups where all vertex and edge groups are infinite cyclic); these are simple enough to study while also providing a rich class of examples of $C^*$-algebras.

As a demonstration of this, we turn to the study of Kirchberg algebras (simple, separable, nuclear and purely infinite $C^*$-algebras). By the celebrated Kirchberg--Phillips theorem (\cite{Kirchberg, Phillips}), stable Kirchberg algebras satisfying the Universal Coefficient Theorem (UCT) are classified completely by their $K$-theory, making this an interesting class of $C^*$-algebras to study. This classification result has led to the development of concrete models for Kirchberg algebras, such as Spielberg's hybrid graph algebras \cite{Spielberg_hybrid_graphs} and Katsura's topological graph algebras \cite{Katsura}, which have in turn led to a deeper understanding of Kirchberg algebras (see for example \cite{Kirchberg_actions}).

In this paper, we find sufficient conditions on a directed graph of infinite cyclic groups for its $C^*$-algebra to be a UCT Kirchberg algebra, and we also provide a formula for the $K$-theory of these algebras. Using these results, we are able to prove the following theorem, which provides two additional models of Kirchberg algebras.

\begin{thm}[Theorem~\ref{thm:dgog Kirchberg}] \label{thm:B}
    Let $A$ be a stable UCT Kirchberg algebra. There exists a countable, row-finite directed graph of infinite cyclic groups $\GG_+ = (\Gamma_+, G)$ with no sources, such that $A \cong C^*(\GG_+) \cong C_0(\partial X_{\GG_+, x}) \rtimes \pi_1(\GG_+, x)$ for any $x \in \Gamma^0$.
\end{thm}

Notably, we show how to construct any stable UCT Kirchberg algebra as a crossed product $C^*$-algebra. Crossed product $C^*$-algebras, particularly those arising from a group action on a commutative $C^*$-algebra, are a fundamental class of examples of $C^*$-algebras, and have been greatly studied from the perspective of determining conditions under which they are Kirchberg algebras (\cite{purely_infinite_AD, crossed_products_classifiability, purely_infinite_JR, purely_infinite_LS}). It is known that all stable UCT Kirchberg algebras are crossed products by $A\T$-algebras \cite{Rordam}, but the construction is not concrete. Theorem~\ref{thm:B} is, to the best of our knowledge, the first concrete construction of (stable) UCT Kirchberg algebras as crossed products. We hope that this might help develop our understanding of Kirchberg algebras, as well as promote further research into the modelling of classifiable $C^*$-algebras by crossed products.

\vspace{0.2cm}

\textbf{Acknowledgements.} The author would like to thank his PhD supervisor Nathan Brownlowe for helpful conversations and comments during the preparation of this paper. The author was supported by an Australian Government Research Training Program Stipend Scholarship.

\renewcommand*{\thethm}{\arabic{thm}}
\numberwithin{thm}{section}
\section{Background} \label{sec:background}

In this section we collect the background material required for the paper. We recall definitions about graphs and trees in Section~\ref{subsec:background graphs}, left cancellative small categories and their $C^*$-algebras in Section~\ref{subsec:background LCSCs}, and groupoids of dynamical origin in Section~\ref{subsec:background groupoids}. Finally in Section~\ref{subsec:background gogs} we recount some of the theory of (undirected) graphs of groups.

\subsection{Graphs and trees} \label{subsec:background graphs}

We start by defining the notions of graph that we will use throughout this paper.

\begin{defin}
    A \textit{directed graph} $E = (E^0, E^1, r, s)$ consists of countable sets of \textit{vertices} $E^0$ and \textit{edges} $E^1$, along with \textit{range} and \textit{source} maps $r, s \colon E^1 \to E^0$. We say that $E$ is \textit{row-finite} if $r^{-1}(v)$ is finite for all $v \in E^0$ (that is, each vertex receives finitely many edges), and $E$ \textit{has no sources} if $r^{-1}(v) \neq \emptyset$ for all $v \in E^0$.
\end{defin}

\begin{defin}
    An \textit{undirected graph} (or simply a \textit{graph}) $\Gamma = (\Gamma^0, \Gamma^1, r, s)$ is a directed graph equipped with an ``edge reversal" map $e \mapsto \overline{e}$ on $\Gamma^1$ such that for all $e \in \Gamma^1$, we have
    \[ \overline{e} \neq e, \quad \overline{\overline{e}} = e, \quad \text{and} \quad s(e) = r(\overline{e}). \]
    An \textit{orientation} of $\Gamma$ is a subset $\Gamma_+^1$ of $\Gamma^1$ containing exactly one element of $\{e, \overline{e}\}$ for each $e \in \Gamma^1$.
\end{defin}

If $\Gamma_+^1$ is an orientation of the graph $\Gamma$, then we can consider the directed graph $\Gamma_+ = (\Gamma^0, \Gamma_+^1, r, s)$, where the range and source maps are inherited from $\Gamma$. The original graph $\Gamma$ can be recovered from $\Gamma_+$ by including a ``reversal" $\ebar$ for each $e \in \Gamma_+^1$; we call $\Gamma$ the \textit{underlying undirected graph} of $\Gamma_+$.

A \textit{path (of length $n$)} in a directed graph $E$ is a sequence of edges $e_1 e_2 \dots e_n$ in $E^1$ satisfying $s(e_i) = r(e_{i+1})$ for $1 \leq i \leq n-1$ (in the case that $n=0$, then a path is just a single vertex $v \in E^0$). We note that we are using the ``Australian" convention for path notation, to be consistent with \cite{BMPST} and \cite{BSTW}. For a path $\lambda = e_1 e_2 \dots e_n$ in $E$, we write $\abs{\lambda} = n$ for the length of $\lambda$. We write $E^n$ for the set of paths in $E$ of length $n$, and we write $v E^n$ for the subset of $E^n$ consisting of paths with range $v \in E^0$. We write $E^*$ for the set of all paths in $E$.

For (undirected) graphs, we have the additional notion of a \textit{reduced} path, which is a path $e_1 e_2 \dots e_n$ such that $e_{i+1} \neq \ebar_i$ for $1 \leq i < n$ (that is, there is no ``back-tracking" in the path). An undirected graph is \textit{connected} if there is a (reduced) path between any two vertices, and a directed graph is connected if its underlying undirected graph is connected. All directed and undirected graphs in this paper will be connected unless specified otherwise.

A \textit{tree} is a graph such that there is a unique reduced path from any vertex to any other vertex. A \textit{directed tree} is a directed graph whose underlying undirected graph is a tree.

We will be considering the following natural topological spaces arising from directed graphs and trees:
\begin{enumerate}[(1)]
    \item Let $E$ be a directed graph. We write $E^\infty$ for the space of infinite paths in $E$ (that is, infinite sequences of edges $e_1 e_2 \dots$ in $E^1$ with $s(e_i) = r(e_{i+1})$ for all $i \geq 1$) with topology generated by finite unions, intersections and set differences of subsets of the form $\lambda E^\infty$, $\lambda \in E^*$.
    
    \item Let $T$ be a directed tree. We can define an equivalence relation on the set of infinite paths in $T$ as follows. We say that two infinite paths $e_1 e_2 \dots$ and $f_1 f_2 \dots$ in $T$ are \textit{shift-tail equivalent} if there exist positive integers $N, M$ such that $e_{N+i} = f_{M+i}$ for all $i \geq 0$. The \textit{boundary} $\partial T$ of $T$ is the space of shift-tail equivalence classes of infinite paths in $T$, with topology generated by finite unions, intersections and set differences of subsets of the form $Z(v) := \{[\lambda] : \lambda \in vT^\infty\}$, $v \in T^0$.
\end{enumerate}

\subsection{Left cancellative small categories and their $C^*$-algebras} \label{subsec:background LCSCs}

We now recall definitions relating to left cancellative small categories and their $C^*$-algebras.

\begin{defin}
    A small category $\Lambda$ is \textit{left cancellative} if $\alpha \beta = \alpha \gamma \implies \beta = \gamma$ for all $\alpha, \beta, \gamma \in \Lambda$ with $s(\alpha) = r(\beta) = r(\gamma)$. We will refer to left cancellative small categories as \textit{LCSC}s for short.
\end{defin}

For any $\alpha \in \Lambda$, we will write $\alpha\Lambda$ to denote the set of morphisms in $\Lambda$ of the form $\alpha\lambda$, $\lambda \in \Lambda$.

\begin{defin}
    An LCSC $\Lambda$ is \textit{finitely aligned} if for every $\alpha, \beta \in \Lambda$ there is a finite set $F \subseteq \Lambda$ such that $\alpha \Lambda \cap \beta \Lambda = \bigcup_{\gamma \in F} \gamma \Lambda$. If $F$ can always be chosen to have at most one element (that is, for every $\alpha, \beta \in \Lambda$ either $\alpha \Lambda \cap \beta \Lambda = \emptyset$, or there is $\gamma \in \Lambda$ such that $\alpha \Lambda \cap \beta \Lambda = \gamma \Lambda$), then $\Lambda$ is \textit{singly aligned}.
\end{defin}

\begin{defin} \label{def:LCSC order}
    Let $\Lambda$ be a LCSC. Define a relation $\leq$ on $\Lambda$ by $\alpha \leq \beta$ if $\beta \in \alpha\Lambda$ (this relation is reflexive and transitive but not necessarily antisymmetric, so it is not a partial order in general). A \textit{common extension} of $\alpha, \beta \in \Lambda$ is an element $\gamma \in \Lambda$ such that $\alpha, \beta \leq \gamma$, and a common extension $\gamma$ of $\alpha, \beta$ is \textit{minimal} if there is no common extension $\delta$ of $\alpha$ and $\beta$ such that $\delta \leq \gamma$ and $\gamma \not\leq\delta$. Write $\alpha \vee \beta$ for the set of minimal common extensions of $\alpha$ and $\beta$.
\end{defin}

\begin{defin}
    Let $\Lambda$ be a LCSC, and let $C \subseteq \Lambda$. A subset $E \subseteq C$ is \textit{exhaustive} if for every $\alpha \in C$, there exists $\beta \in E$ such that $\alpha\Lambda \cap \beta\Lambda \neq \emptyset$.
\end{defin}

In \cite{LCSC}, Spielberg defined the \textit{Cuntz--Krieger algebra} for a LCSC. In the finitely aligned case, we have the following presentation by generators and relations, which we use here as an alternate definition:

\begin{defin}[{\cite[Theorem~10.15(i)]{LCSC}}]
    Let $\Lambda$ be a finitely aligned LCSC. A \textit{Cuntz--Krieger $\Lambda$-family} is a collection $\{S_\lambda : \lambda \in \Lambda\}$ satisfying
    \begin{enumerate}[(S1)]
        \item $S_\lambda^* S_\lambda = S_{s(\lambda)}$;
        \item $S_\lambda S_\mu = S_{\lambda\mu}$ if $s(\lambda) = r(\mu)$;
        \item $S_\lambda S_\lambda^* S_\mu S_\mu^* = \bigvee_{\nu \in \lambda \vee \mu} S_\nu S_\nu^*$;
        \item $S_v = \bigvee_{\lambda \in F} S_\lambda S_\lambda^*$ if $F \subseteq v\Lambda$ is a finite exhaustive set.
    \end{enumerate}
    The \textit{Cuntz--Krieger algebra $\OO(\Lambda)$} of $\Lambda$ is the universal $C^*$-algebra generated by a Cuntz--Krieger $\Lambda$-family.
\end{defin}

\begin{remark} \label{remark:LCSC family}
    Relations (S1) and (S2) imply that for each object $v \in \Lambda^0$ (which we identify with the identity morphism $1_v \in \Lambda$), the element $S_v = S_{1_v}$ is a projection (take $\lambda = \mu = 1_v$ in the relations). Thus (S1) implies that each $S_\lambda$, $\lambda \in \Lambda$ is a partial isometry.
\end{remark}

\subsection{Groupoids of dynamical origin} \label{subsec:background groupoids}

Here we recall the definitions of groupoid and inverse semigroup actions on spaces, and of groupoids that encode these actions. We refer the reader to the book \cite{Paterson} for a more in-depth discussion of these objects.

\begin{defin}
    A \textit{groupoid} $\mathscr{G}$ is a small category with inverses: that is, it has a set of units $\mathscr{G}^{(0)}$ and morphisms $\mathscr{G}$ such that:
    \begin{enumerate}
        \item each morphism $g \in \mathscr{G}$ has a \textit{range} $r(g) \in \mathscr{G}^{(0)}$ and a source $s(g) \in \mathscr{G}^{(0)}$;
        \item there is a partially defined composition on $\mathscr{G}$, where a pair of morphisms $(g, h) \in \mathscr{G} \times \mathscr{G}$ is composable if and only if $s(g) = r(h)$;
        \item each unit $x \in \mathscr{G}^{(0)}$ has a corresponding \textit{identity morphism} $x \in \mathscr{G}$ (to which it will be identified) satisfying $xg = g$ for all $g \in \mathscr{G}$ with range $x$, and $gx = g$ for all $g \in \mathscr{G}$ with source $x$; and
        \item each morphism $g \in \mathscr{G}$ has a (unique) \textit{inverse} $g^{-1} \in \mathscr{G}$ satisfying $gg^{-1} = r(g)$ and $g^{-1}g = s(g)$.
    \end{enumerate}
    
    A \textit{topological groupoid} is a groupoid which additionally carries a topology such that composition and inversion are continuous. All topological groupoids we consider will be \textit{Hausdorff \'etale}, meaning that they have a Hausdorff topology and the range map $r:\mathscr{G} \to \mathscr{G}$ is a local homeomorphism.
\end{defin}

If $x$ and  $y$ are units in $\mathscr{G}^{(0)}$, we will write $x\mathscr{G}$ to denote the set of morphisms in $\mathscr{G}$ with range $x$, and $x\mathscr{G}y$ to denote the set of morphisms in $\mathscr{G}$ with range $x$ and source $y$. If $U \subseteq \mathscr{G}^{(0)}$ is a subset of the unit space of $\mathscr{G}$, we write $\mathscr{G}|_U$ for the \textit{restriction of $\mathscr{G}$ to $U$}, which is the subgroupoid of $\mathscr{G}$ consisting of morphisms whose range and source are both in $U$.

\begin{defin}
    An \textit{inverse semigroup} is a semigroup $\SS$ such that for each element $s \in \SS$ there is a unique element $s^* \in \SS$ such that $ss^*s = s$ and $s^*ss^* = s^*$. A \textit{zero element} of $\SS$ is an element $0 \in S$ such that $s0 = 0s = 0$ for all $s \in \SS$.
\end{defin}

Both groupoids and inverse semigroups act on topological spaces by \textit{partial homeomorphisms}, which we now define.

\begin{defin}
    Let $X$ be a locally compact Hausdorff space. A \textit{partial homeomorphism} $\varphi$ of $X$ is a homeomorphism from a subspace $D(\varphi)$ of $X$ (the \textit{domain} of $\varphi$) to another subspace $R(\varphi)$ of $X$ (the \textit{range} of $\varphi$). For two partial homeomorphisms $\varphi, \psi$ of $X$, the composition $\varphi \circ \psi$ is the partial homeomorphism with domain $\psi^{-1}(R(\psi) \cap D(\varphi))$ and range $\varphi(R(\psi) \cap D(\varphi))$ satisfying $(\varphi \circ \psi)(x) = \varphi(\psi(x))$ for $x \in D(\varphi \circ \psi)$.
\end{defin}

\begin{defin}[Actions]
    Let $X$ be a locally compact Hausdorff space.
    \begin{enumerate}[(1)]
        \item An action of a groupoid $\mathscr{G}$ on $X$ consists of a surjective \textit{anchor map} $\tau \colon X \to \mathscr{G}^{(0)}$, and a collection of partial homeomorphisms $\{\varphi_g : g \in \mathscr{G}\}$ of $X$ such that
        \begin{enumerate}[(i)]
            \item for all $g \in \mathscr{G}$, we have $D(\varphi_g) = \tau^{-1}(s(g))$ and $R(\varphi_g) = \tau^{-1}(r(g))$;
            \item for all $g \in \mathscr{G}$, we have $\varphi_{g^{-1}} = \varphi_g^{-1}$; and
            \item for all $g, h \in \mathscr{G}$ with $s(g) = r(h)$, we have $\varphi_{gh} = \varphi_g \circ \varphi_h$.
        \end{enumerate}
        We will often write $g \cdot x$ instead of $\varphi_g(x)$, for $g \in \mathscr{G}$ and $x \in X$.
        
        \item An action of an inverse semigroup $\SS$ on $X$ is a collection of partial homeomorphisms $\{\varphi_s : s \in \SS\}$ of $X$ such that for all $s, t \in \SS$, we have that $\varphi_{s^*} = \varphi_s^{-1}$ and $\varphi_{st} = \varphi_s \circ \varphi_t$. As with groupoid actions, 
        we will often write $s \cdot x$ instead of $\varphi_s(x)$, for $s \in \SS$ and $x \in X$.
    \end{enumerate}
\end{defin}

\begin{defin}[Groupoids associated to actions] Let $X$ be a locally compact Hausdorff space.
    \begin{enumerate}[(1)]
        \item Suppose that the groupoid $\mathscr{G}$ acts on $X$. The \textit{transformation groupoid} of the action $\mathscr{G} \curvearrowright X$ is the set $\mathscr{G} \ltimes X := \{(g,x) : g \in \mathscr{G},\ x \in X,\ \tau(x) = s(g)\}$, where for $(g, x) \in \mathscr{G} \ltimes X$:
        \begin{enumerate}[(i)]
            \item $r(g, x) = x$ and $s(g, x) = g \cdot x$;
            \item if $(h, y) \in \mathscr{G}$ satisfies $y = g \cdot x$, then $(h,y)(g,x) = (hg,x)$; and
            \item $(g,x)^{-1} = (g^{-1}, g \cdot x)$.
        \end{enumerate}
        The topology of $\mathscr{G} \ltimes X$ is the subspace topology induced by the product topology on $\mathscr{G} \times X$ (where $\mathscr{G}$ has the discrete topology).
         
        \item Suppose that the inverse semigroup $\SS$ acts on $X$. Write $S*X = \{(s,x) : s \in \SS,\ x \in D_{s^*s}\}$ and define an equivalence relation on $S*X$ by $(s,x) \sim (t,y)$ if and only if $x=y$ and there is an idempotent $e \in \SS$ such that $x \in D_e$ and $se = te$. The \textit{groupoid of germs} of the action $\SS \curvearrowright X$ is the set $\SS \ltimes X = S*X/\sim$ with composition given by $[s,x][t,y] = [st,y]$ if $x = t \cdot y$, and with inverse given by $[s,x]^{-1} = [s^*, s \cdot x]$. The topology of $\SS \ltimes X$ is generated by open sets of the form $[s, U] := \{[s,x] : x \in U\}$ where $U \subseteq D_{s^* s} \subseteq X$ is an open set.
    \end{enumerate}
\end{defin}

To any Hausdorff \'etale groupoid $\mathscr{G}$ one can associate a $C^*$-algebra $C^*(\mathscr{G})$; we refer the reader to \cite{Renault_groupoid_book} for details. Here we just note that for an action of a group $G$ on a locally compact Hausdorff space $X$, the $C^*$-algebra $C^*(G \ltimes X)$ of the associated transformation groupoid is isomorphic to the \textit{full crossed product} $C_0(X) \rtimes G$ for the induced action of $G$ on $C_0(X)$.

\subsection{Graphs of groups} \label{subsec:background gogs}

In this subsection we briefly recall the theory of (undirected) graphs of groups. For a more detailed account, we refer the reader to \cite{BMPST}.

\begin{defin}
    A \textit{graph of groups} $\GG = (\Gamma, G)$ is a graph $\Gamma$ along with:
    \begin{enumerate}
        \item a \textit{vertex group} $G_v$ for each $v \in \Gamma^0$;
        \item an \textit{edge group} $G_e$ for each $e \in \Gamma^1$, such that $G_e = G_{\overline{e}}$ for all $e \in \Gamma^1$; and
        \item a monomorphism $\alpha_e \colon G_e \to G_{r(e)}$ for each $e \in \Gamma^1$.
    \end{enumerate}
\end{defin}

Graphs of groups naturally arise as ``quotient objects" of group actions on trees: if a group $G$ acts on a tree $T$ without inversions (that is, there are no $g \in G$, $e \in T^1$ with $g \cdot e = \overline{e}$), the action induces a graph of groups as follows. We let $\Gamma = G\backslash T$ be the quotient graph of $T$ under the action of $G$: that is, the graph whose vertices and edges are the vertex and edge orbits of $T$ respectively, and the range and source maps are induced from those of $T$. For each $v \in \Gamma^0$ we pick a lift $v' \in T^0$, and define the vertex group $G_v$ to be the stabiliser subgroup of $G$ at $v' \in T^0$; similarly for each $e \in \Gamma^1$ we pick a lift $e' \in T^1$, and define the edge group $G_e$ to be the stabiliser subgroup of $G$ at $e' \in T^1$. Finally, for each $e \in \Gamma^1$ we define the monomorphism $\alpha_e$ as follows. Write $v = r(e)$. Then $v'$ is in the same vertex orbit as $r(e')$, so there is some $g \in G$ such that $g \cdot r(e') = v'$. Now, each $h \in G_e$ fixes $e'$ and thus also fixes $r(e')$; then $ghg^{-1}$ fixes $v'$, so $ghg^{-1} \in G_v = G_{r(e)}$. Hence we can define $\alpha_e\colon G_e \to G_{r(e)}$ by $h \mapsto ghg^{-1}$.

In fact, from any graph of groups $\GG$, one can construct a group action on a tree whose induced graph of groups is isomorphic to $\GG$ (in the sense that there is an isomorphism of graphs, and there are isomorphisms of corresponding vertex and edge groups which preserve the monomorphisms); this is part of what is known as the Fundamental Theorem of Bass--Serre theory, which we quote as Theorem~\ref{thm:Bass-Serre}. Before we can state this theorem, we recount some more theory concerning graphs of groups. We start by discussing what is called the \textit{fundamental groupoid} of $\GG$, a notion introduced in \cite{fundamental_groupoid}.

\begin{defin}
    Let $\GG = (\Gamma, G)$ be a graph of groups. The \textit{fundamental groupoid $F(\GG)$} of $\GG$ is the groupoid defined by the following presentation. The vertex set of $F(\GG)$ is $\Gamma^0$; the morphisms of $F(\GG)$ are generated by $\Gamma^1 \sqcup \bigsqcup_{v \in \Gamma^0} G_v$, where the range and source of $e \in \Gamma^1$ in $F(\GG)$ are as they are in $\Gamma$, and for any $v \in \Gamma^0$ and $g \in G_v$ we have $r(g) = s(g) = v$; and the relations consist of the relations in each vertex group $G_v$, along with
    \begin{enumerate}[(i)]
        \item $\overline{e} = e^{-1}$ for all $e \in \Gamma^1$; and
        \item $\alpha_e(g) e = e \alpha_{\overline{e}}(g)$ for all $e \in \Gamma^1$ and $g \in G_e$.
    \end{enumerate}
\end{defin}

Elements of the fundamental groupoid can be represented by \textit{$\GG$-words}, which we will now define.

\begin{defin}
    Let $\GG = (\Gamma, G)$ be a graph of groups. A \textit{$\GG$-word (of length $n$)} is a sequence of the form $g_1 e_1 g_2 e_2 \dots g_n e_n g_{n+1}$, where $e_1 \dots e_n$ is a path in $\Gamma$, $g_i \in G_{r(e_i)}$ for $i = 1, \dots, n$, and $g_{n+1} \in G_{s(e_n)}$ (in the case where $n=0$, the $\GG$-word is just an element $g_1$ of some vertex group $G_v$). We set the range and source of a $\GG$-word to be $r(e_1)$ and $s(e_n)$ respectively (in the case where $n=0$, we set the range and source to be the vertex whose group the element $g_1$ belongs to).
\end{defin}

We will often identify a $\GG$-word with its image in $F(\GG)$. However, note that multiple $\GG$-words can represent the same element of $F(\GG)$. In order to find a unique representative for each element of $F(\GG)$, we need the concept of a \textit{normalised} word.

\begin{defin} \label{def:Sigma-normalised G-word}
    Let $\GG = (\Gamma, G)$ be a graph of groups. A \textit{set of transversals for $\GG$} is a set $\Sigma = \{\Sigma_e : e \in \Gamma^1\}$, where each $\Sigma_e$ is a transversal of $G_{r(e)}/\alpha_e(G_e)$ which includes the identity element $1_{r(e)} \in G_{r(e)}$. A $\GG$-word $g_1 e_1 \dots g_n e_n g_{n+1}$ is called \textit{$\Sigma$-normalised} if $g_i \in \Sigma_{e_i}$ for $i = 1, \dots, n$, and the $\GG$-word does not have any subsequence of the form $\overline{e} 1_{r(e)} e$.
\end{defin}

\begin{prop}[\mbox{\cite[Theorem]{fundamental_groupoid}}] \label{prop:F(G) Sigma-normalised rep}
    Let $\GG = (\Gamma, G)$ be a graph of groups, and let $\Sigma$ be a set of transversals for $\GG$. Each element of $F(\GG)$ is represented by a unique $\Sigma$-normalised word.
\end{prop}

The fundamental groupoid $F(\GG)$ acts naturally on the following graph \cite[Section~2.4]{BMPST}.

\begin{defin} \label{def:W_G}
    Let $\GG = (\Gamma, G)$ be a graph of groups. Define the graph $W_\GG$ as follows. The vertex set of $W_\GG$ is
    \[ W_\GG^0 = \bigsqcup_{v,w \in \Gamma^0} v F(\GG) w / G_w = \{\gamma G_w : v,w \in \Gamma^1,\ \gamma \in v F(\GG) w\}; \]
    and for $\gamma \in v F(\GG) w$ and $\gamma' \in v' F(\GG) w'$, there is an edge in $W_\GG$ from $\gamma' G_{w'}$ to $\gamma G_w$ if and only if $v = v'$ and $\gamma^{-1} \gamma' \in G_v e G_w$ for some $e \in v \Gamma^1 w$.
\end{defin}

We note that $W_\GG$ is not connected in general. Each connected component of $W_\GG$ consists of the vertices $
\bigsqcup_{w \in \Gamma^0} x F(\GG) w / G_w$ for some unique, fixed $x \in \Gamma^0$; by \cite[Theorem~1.17]{Bass}, each of these components are trees, so $W_\GG$ is a forest.

The fundamental groupoid $F(\GG)$ acts on $W_\GG$ in the following way. Let $v, w \in \Gamma^0$ and take $\gamma' \in v F(\GG) w$. An element $\gamma \in F(\GG)$ acts on $\gamma' G_w \in W_\GG^0$ if and only if $v = s(\gamma)$, in which case we have $\gamma \cdot \gamma' G_w = \gamma\gamma' G_w$. Since $W_\GG$ is a forest, the action on the vertices induces an action on the graph, and one can check that this action is without inversions.

We can now define the group action on a tree induced by a graph of groups.

\begin{defin}
    Let $\GG = (\Gamma, G)$ be a graph of groups, and let $x \in \Gamma^0$. The \textit{fundamental group of $\GG$ based at $x$}, which we denote by $\pi_1(\GG, x)$, is the isotropy group of $F(\GG)$ at $x$. The \textit{Bass--Serre tree of $\GG$ based at $x$}, which we denote by $X_{\GG,x}$, is the connected component in $W_\GG$ of the vertex $G_x$ (equivalently, it is the restriction of $W_\GG$ to the vertices $\bigsqcup_{w \in \Gamma^0} x F(\GG) w/ G_w$).
\end{defin}

The action of $F(\GG)$ on $W_\GG$ restricts to an action of the group $\pi_1(\GG, x)$ on the tree $X_{\GG,x}$. The significance of this action comes from the Fundamental Theorem of Bass--Serre theory, which we now state.

\begin{thm} \label{thm:Bass-Serre}
    Let $\GG = (\Gamma, G)$ be a graph of groups, and let $x \in \Gamma^0$. The induced graph of groups for the action of $\pi_1(\GG, x)$ on $X_{\GG, x}$ is isomorphic to $\GG$. Conversely, if $\GG$ is induced from an action (without inversions) of a group $G$ on the tree $T$, then there is an isomorphism of groups $\pi_1(\GG, x) \cong G$ and an equivariant isomorphism of trees $X_{\GG, x} \cong T$.
\end{thm}

\begin{remark} \label{rmk:B-S tree lifts}
    Let $\GG = (\Gamma, G)$ be a graph of groups, and let $x \in \Gamma^0$. For any vertex $v \in \Gamma^0$, the lifts of $v$ in the Bass--Serre tree $X_{\GG, x}$ are precisely the vertices
    \[ x F(\GG) v / G_v = \{ \gamma G_v : \gamma \in x F(\GG) v\}. \]
    For any edge $e \in \Gamma^1$, the lifts of $e$ in $X_{\GG, x}$ are the edges with range $\gamma G_{r(e)}$ for some $\gamma \in x F(\GG) r(e)$ and with source $\gamma g e G_{s(e)}$ for some $g \in G_{r(e)}$.
\end{remark}

\section{Directed graphs of groups} \label{sec:dgogs}

In this section, we introduce and develop some theory for \textit{directed graph of groups}. We start by defining directed graphs of groups in Section~\ref{subsec:dgogs}, as well as explaining how they relate to group actions on directed trees. In Section~\ref{subsec:word category}, we study a canonical (left cancellative small) category associated to a directed graph of groups, which we call its \textit{word category}. Finally, in Section~\ref{subsec:F(G) boundary action} we study an action, canonically associated to a directed graph of groups, of a groupoid on a directed forest.

\subsection{Directed graphs of groups} \label{subsec:dgogs}

We start by defining directed graphs of groups, and then briefly discuss a `directed version' of Bass--Serre theory.

\begin{defin}
    A \textit{directed graph of groups} $\GG_+ = (\Gamma_+, G)$ is a graph of groups $\GG = (\Gamma, G)$ where the graph $\Gamma$ additionally carries an orientation $\Gamma_+^1$. We call $\GG$ the \textit{underlying graph of groups of $\GG_+$}.
\end{defin}

\begin{remark}
    We note that directed graphs of groups have also been studied in \cite{levi_categories}, where they were called `graphs of groups with a given orientation'.
\end{remark}

Just as graphs of groups are induced from group actions on trees, \textit{directed} graphs of groups are induced from group actions on \textit{directed} trees. Indeed, if a group $G$ acts on a directed tree $T_+$, then it acts on the underlying undirected tree $T$ without inversion; this induces a graph of groups $\GG = (\Gamma, G)$, and we can give $\Gamma$ the orientation induced by the orientation on $T$, yielding a directed graph of groups $\GG_+ = (\Gamma_+, G)$. Moreover, we can recover the action $G \curvearrowright T_+$ from $\GG_+$, like in the undirected case: the acting group is still the fundamental group $\pi_1(\GG, x)$, and the directed tree is the Bass--Serre tree $X_{\GG, x}$ with a suitable orientation, which we now describe.

\begin{defin}
    Let $\GG_+ = (\Gamma_+, G)$ be a directed graph of groups, and choose a base vertex $x \in \Gamma^0$. The \textit{directed Bass--Serre tree $X_{\GG_+,x}$} of $\GG_+$ based at $x$ is the directed tree with $X_{\GG,x}$ as the underlying undirected tree, and where an edge from $\gamma'G_w$ to $\gamma G_v$ is in the orientation if and only if $\gamma^{-1}\gamma' \in G_v e G_w$ for some $e \in v\Gamma_+^1 w$.
\end{defin}

We now state and prove the directed analogue of the fundamental theorem of Bass--Serre theory (Theorem~\ref{thm:Bass-Serre}).

\begin{thm} \label{thm:directed Bass-Serre}
    Let $\GG_+ = (\Gamma_+, G)$ be a directed graph of groups, and choose a base vertex $x \in \Gamma^0$. The induced directed graph of groups for the action of $\pi_1(\GG,x)$ on $X_{\GG_+,x}$ is isomorphic to $\GG_+$. Conversely, if $\GG_+$ is induced from an action of a group $G$ on a directed tree $T_+$, then there is an isomorphism of groups $\pi_1(\GG, x) \cong G$ and an equivariant isomorphism of directed trees $X_{\GG_+, x} \cong T_+$.
\end{thm}

\begin{proof}
    In light of Theorem~\ref{thm:Bass-Serre}, all that we need to check is that the orientation of the quotient graph $\pi_1(\GG, x) \backslash X_{\GG_+, x}$ matches the orientation of $\Gamma_+$. Take $e \in \Gamma_+^1$, and take some lift of $e$ in $X_{\GG, x}$ (by Remark~\ref{rmk:B-S tree lifts}, the lift goes from $\gamma g e G_{s(e)}$ to $\gamma G_{r(e)}$ for some $\gamma \in x F(\GG) r(e)$ and $g \in G_{r(e)}$). Remark~\ref{rmk:B-S tree lifts} also gives that $\gamma G_{r(e)}$ is a lift of the vertex $r(e)$ and $\gamma g e G_{s(e)}$ is a lift of the vertex $s(e)$. Since $\gamma^{-1}(\gamma g e) = g e 1_{s(e)} \in G_{r(e)} e G_{s(e)}$, there is an edge in the orientation of $X_{\GG_+, x}$ from $\gamma g e G_{s(e)}$ to $\gamma G_{r(e)}$, which matches the direction of $e$ in $\Gamma_+^1$. This proves the theorem.
\end{proof}

\pagebreak
We end this subsection with some examples.

\begin{egs}~
    \begin{enumerate}[(1)]
        \item Consider the following directed graph of groups (where $1$ denotes the trivial group, and the monomorphisms are the trivial inclusions).
        \[\begin{tikzpicture}[scale=2.5]
            \node[vertex, label=below:{$v$}, label=above:{$\Z_2$}] (v) at (0,0) {};
            \node[open-vertex, label=below:{$w$}, label=above:{$\Z_3$}] (w) at (1,0) {};
            \draw[-stealth] (w) -- node[midway, label=below:{$e$}, label=above:{$1$}] {} (v);
        \end{tikzpicture}\]
        We have transversals $\Sigma_e = \Z_2 = \{1, a\}$ and $\Sigma_{\overline{e}} = \Z_3 = \{1, b, b^2\}$, and $\GG$-words have edge sequence alternating between $e$ and $\overline{e}$. The fundamental group $\pi_1(\GG, v)$ is the free product $\Z_2 * \Z_3$ with generators $a$ and $eb\ebar$ of order 2 and 3 respectively, and the directed Bass--Serre tree is the following.
        
        
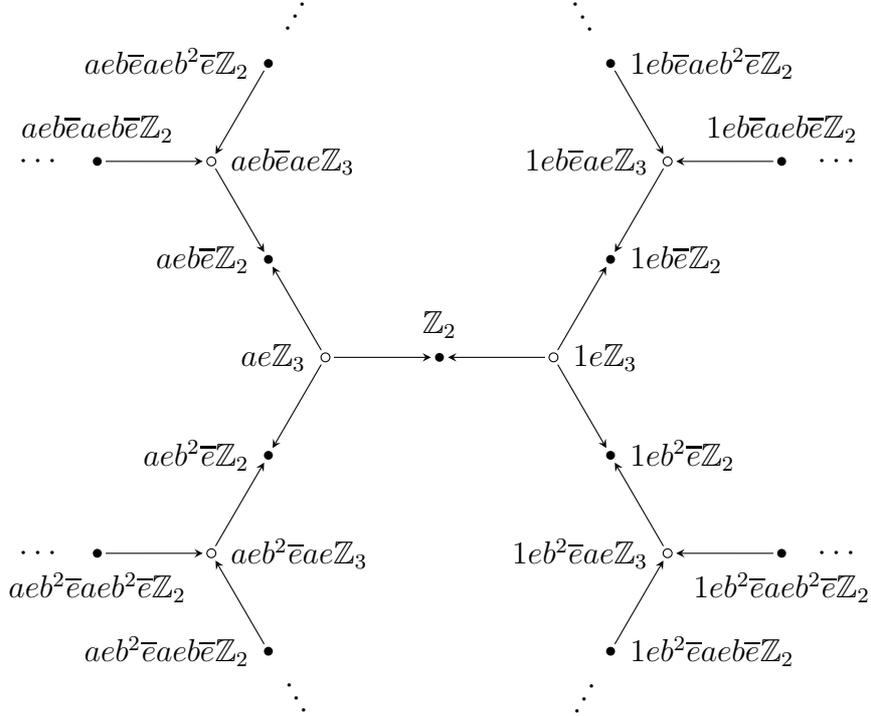
\[\begin{tikzpicture}[scale=1.5]
            \node[vertex, label=above:{$\Z_2$}] (v) at (0,0) {};
            
            \node[open-vertex, label=right:{$1e\Z_3$}]
                (0e) at (1,0) {};
            \node[open-vertex, label=left:{$ae\Z_3$}]
                (1e) at (-1,0) {};
            
            \node[vertex, label=right:{$1eb\overline{e}\Z_2$}]
                (0e1f) at ($(1,0) + (60:1)$) {};
            \node[vertex, label=right:{$1eb^2\overline{e}\Z_2$}]
                (0e2f) at ($(1,0) + (-60:1)$) {};
            
            \node[open-vertex, label=left:{$1eb\overline{e}ae\Z_3$}]
                (0e1f1e) at ($(1,0) + (60:2)$) {};
            \node[open-vertex, label=left:{$1eb^2\overline{e}ae\Z_3$}]
                (0e2f1e) at ($(1,0) + (-60:2)$) {};
            
            \node[vertex, label=right:{$1eb\overline{e}aeb^2\overline{e}\Z_2$}]
                (0e1f1e2f) at ($(0e1f1e) + (120:1)$) {};
            \node at ($(0e1f1e2f) + (120:0.5)$) {\rotatebox{120}{$\cdots$}};
            \node[vertex, label=above:{$1eb\overline{e}aeb\overline{e}\Z_2$}]
                (0e1f1e1f) at ($(0e1f1e) + (0:1)$) {};
            \node at ($(0e1f1e1f) + (0:0.5)$) {$\cdots$};
                
            \node[vertex, label=right:{$1eb^2\overline{e}aeb\overline{e}\Z_2$}]
                (0e2f1e1f) at ($(0e2f1e) + (-120:1)$) {};
            \node at ($(0e2f1e1f) + (-120:0.5)$) {\rotatebox{-120}{$\cdots$}};
            \node[vertex, label=below:{$1eb^2\overline{e}aeb^2\overline{e}\Z_2$}]
                (0e2f1e2f) at ($(0e2f1e) + (0:1)$) {};
            \node at ($(0e2f1e2f) + (0:0.5)$) {$\cdots$};
            
            \node[vertex, label=left:{$aeb\overline{e}\Z_2$}]
                (1e1f) at ($(-1,0) + (120:1)$) {};
            \node[vertex, label=left:{$aeb^2\overline{e}\Z_2$}]
                (1e2f) at ($(-1,0) + (-120:1)$) {};
            
            \node[open-vertex, label=right:{$aeb\overline{e}ae\Z_3$}]
                (1e1f1e) at ($(-1,0) + (120:2)$) {};
            \node[open-vertex, label=right:{$aeb^2\overline{e}ae\Z_3$}]
                (1e2f1e) at ($(-1,0) + (-120:2)$) {};
                
            \node[vertex, label=left:{$aeb\overline{e}aeb^2\overline{e}\Z_2$}]
                (1e1f1e2f) at ($(1e1f1e) + (60:1)$) {};
            \node at ($(1e1f1e2f) + (60:0.5)$) {\rotatebox{60}{$\cdots$}};
            \node[vertex, label=above:{$aeb\overline{e}aeb\overline{e}\Z_2$}]
                (1e1f1e1f) at ($(1e1f1e) + (180:1)$) {};
            \node at ($(1e1f1e1f) + (180:0.5)$) {$\cdots$};
                
            \node[vertex, label=left:{$aeb^2\overline{e}aeb\overline{e}\Z_2$}]
                (1e2f1e1f) at ($(1e2f1e) + (-60:1)$) {};
            \node at ($(1e2f1e1f) + (-60:0.5)$) {\rotatebox{-60}{$\cdots$}};
            \node[vertex, label=below:{$aeb^2\overline{e}aeb^2\overline{e}\Z_2$}]
                (1e2f1e2f) at ($(1e2f1e) + (180:1)$) {};
            \node at ($(1e2f1e2f) + (180:0.5)$) {$\cdots$};
            
            \draw[-stealth]
                (0e) edge (v)
                (1e) edge (v)
                
                (0e) edge (0e1f)
                (0e) edge (0e2f)
                
                (1e) edge (1e1f)
                (1e) edge (1e2f)
                
                (0e1f1e) edge (0e1f)
                (0e2f1e) edge (0e2f)
                
                (1e1f1e) edge (1e1f)
                (1e2f1e) edge (1e2f)
                
                (0e1f1e1f) edge (0e1f1e)
                (0e1f1e2f) edge (0e1f1e)
                
                (0e2f1e1f) edge (0e2f1e)
                (0e2f1e2f) edge (0e2f1e)
                
                (1e1f1e1f) edge (1e1f1e)
                (1e1f1e2f) edge (1e1f1e)
                
                (1e2f1e1f) edge (1e2f1e)
                (1e2f1e2f) edge (1e2f1e);
        \end{tikzpicture}\]
        
        The fundamental group acts on the tree as follows: the copy of $\Z_2$ in $\pi_1(\GG, v)$ acts by flipping about the vertex $\Z_2$, and the copy of $\Z_3$ acts by rotating about the vertex $1e\Z_3$. The solid vertices in the tree are the lifts of the vertex $v$ from the directed graph of groups, and the unfilled vertices are the lifts of the vertex $w$.
        
        \item Consider the following directed graph of groups.
        \[\begin{tikzpicture}
            \node[vertex, label=right:{$v$}, label=left:{$\Z$}] (v) at (0,0) {};
            \draw[-stealth] ($(v) + ({1-cos(7)},{-sin(7)})$) arc (-173:173:1) node[midway, label={[label distance=-3pt]left:{$e$}}, label={[label distance=-3pt]right:{$\Z$}}] {};
            
            \node at ($(1,0) + (145:1.4)$) {$\times 2$};
            \node at ($(1,0) + (-145:1.4)$) {$\times 1$};
        \end{tikzpicture}\]
        Here, the $\times 2$ and $\times 1$ represent the monomorphisms $\alpha_e$ and $\alpha_\ebar$ respectively, which are defined by $\alpha_e(1) = 2$ and $\alpha_\ebar(1) = 1$. We have transversals $\Sigma_e = \{0,1\}$ and $\Sigma_\ebar = \{0\}$. The fundamental group $\pi_1(\GG, v)$ turns out to be the Baumslag--Solitar group $BS(1,2) = \langle a, b : a b = b^2 a \rangle$, with $a$ corresponding to the element $e \in \pi_1(\GG, v)$ and $b$ corresponding to the generator of the vertex group $\Z$. The directed Bass--Serre tree is the following.
        
        
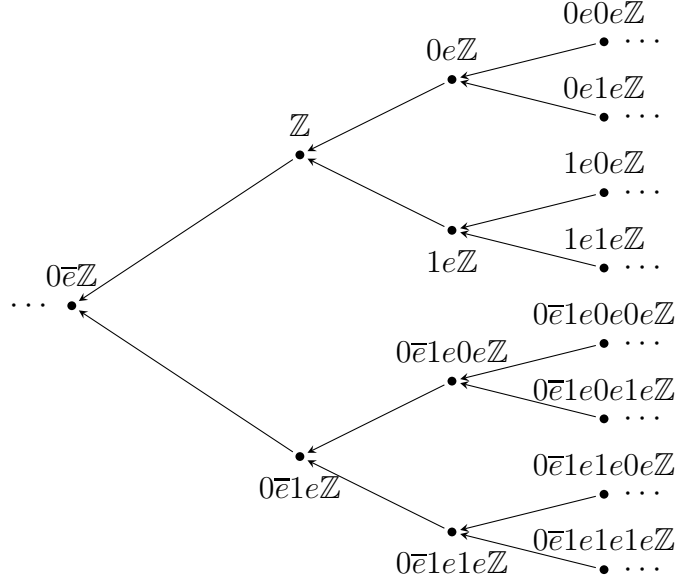
\[\begin{tikzpicture}
            \node[vertex, label=above:{$0\ebar\Z$}, label=left:{$\cdots$}] (0f) at (0,0) {};
            
            \node[vertex, label=above:{$\Z$}] (v) at (3,2) {};
            \node[vertex, label=below:{$0\ebar 1e\Z$}] (0f1e) at (3,-2) {};

            \node[vertex, label=above:{$0e\Z$}] (0e) at ($(v) + (2,1)$) {};
            \node[vertex, label=below:{$1e\Z$}] (1e) at ($(v) + (2,-1)$) {};
            \node[vertex, label=above:{$0\ebar 1e0e\Z$}] (0f1e0e) at ($(0f1e) + (2,1)$) {};
            \node[vertex, label=below:{$0\ebar 1e1e\Z$}] (0f1e1e) at ($(0f1e) + (2,-1)$) {};

            \node[vertex, label=above:{$0e0e\Z$}, label=right:{$\cdots$}] (0e0e) at ($(0e) + (2,0.5)$) {};
            \node[vertex, label=above:{$0e1e\Z$}, label=right:{$\cdots$}] (0e1e) at ($(0e) + (2,-0.5)$) {};
            \node[vertex, label=above:{$1e0e\Z$}, label=right:{$\cdots$}] (1e0e) at ($(1e) + (2,0.5)$) {};
            \node[vertex, label=above:{$1e1e\Z$}, label=right:{$\cdots$}] (1e1e) at ($(1e) + (2,-0.5)$) {};
            
            \node[vertex, label=above:{$0\ebar 1e0e0e\Z$}, label=right:{$\cdots$}] (0f1e0e0e) at ($(0f1e0e) + (2,0.5)$) {};
            \node[vertex, label=above:{$0\ebar 1e0e1e\Z$}, label=right:{$\cdots$}] (0f1e0e1e) at ($(0f1e0e) + (2,-0.5)$) {};
            \node[vertex, label=above:{$0\ebar 1e1e0e\Z$}, label=right:{$\cdots$}] (0f1e1e0e) at ($(0f1e1e) + (2,0.5)$) {};
            \node[vertex, label=above:{$0\ebar 1e1e1e\Z$}, label=right:{$\cdots$}] (0f1e1e1e) at ($(0f1e1e) + (2,-0.5)$) {};
            
            \draw[-stealth]
                (0e) edge (v)
                (1e) edge (v)
                
                (0e0e) edge (0e)
                (0e1e) edge (0e)
                
                (1e0e) edge (1e)
                (1e1e) edge (1e)
                
                (v) edge (0f)
                (0f1e) edge (0f)
                
                (0f1e0e) edge (0f1e)
                (0f1e1e) edge (0f1e)
                
                (0f1e0e0e) edge (0f1e0e)
                (0f1e0e1e) edge (0f1e0e)
                (0f1e1e0e) edge (0f1e1e)
                (0f1e1e1e) edge (0f1e1e);
        \end{tikzpicture}\]
        The fundamental group acts on the tree as follows: the element $a \in \pi_1(\GG, v)$ ``shifts" the tree along the topmost infinite path in the figure above: it sends $0\ebar \Z$ to $\Z$, $\Z$ to $0e\Z$, $0e\Z$ to $0e0e\Z$, and so on. The element $b \in \pi_1(\GG, v)$ fixes the vertex $\Z$ but swaps the edges it receives; more generally it fixes each ``level" of the tree (that is, each set of vertices vertically in line with each other), but permutes the vertices within each level.
    \end{enumerate}
    
\end{egs}

\subsection{Word category} \label{subsec:word category}

In this subsection we introduce and study the \textit{word category} of a directed graph of groups, which generalises the notion of the path category of a directed graph. The results in this subsection will be useful in Sections~\ref{sec:dgog algebra} and \ref{sec:C* Bass-Serre thm}.

\begin{defin}
    Let $\GG_+ = (\Gamma_+, G)$ be a directed graph of groups. A \textit{$\GG_+$-word} is a $\GG$-word where all edges in the word belong to $\Gamma_+^1$. A $\GG_+$-word maps naturally into the fundamental groupoid $F(\GG)$ of the underlying graph of groups. The \textit{word category} of $\GG_+$, which we will denote by $\Lambda_{\GG_+}$, is the subcategory of $F(\GG)$ consisting of the images of $\GG_+$-words.
\end{defin}

Note that the word category is indeed a category since the concatenation of two $\GG_+$-words is again a $\GG_+$-word. Moreover, since $\Lambda_{\GG_+}$ is a subcategory of a groupoid, it must be left cancellative.

\begin{remark}
    This construction of a category from a directed graph of groups yields the same result as the construction in \cite{levi_categories}. Thus, \cite[Theorem~4.17]{levi_categories} tells us that we can recover the directed graph of groups $\GG_+$ from its word category $\Lambda_{\GG_+}$ (up to isomorphism).
\end{remark}

Recall from Section~\ref{subsec:background gogs} that elements of the fundamental groupoid $F(\GG)$ of a graph of groups $\GG$ can be uniquely represented by $\Sigma$-normalised $\GG$-words. For a \textit{directed} graph of groups $\GG_+$, we can likewise define the notion of a $\Sigma$-normalised $\GG_+$-word, and we show that these uniquely represent elements of the word category $\Lambda_{\GG_+}$.

\begin{defin}
    Let $\GG_+ = (\Gamma_+, G)$ be a directed graph of groups. A \textit{set of transversals for $\GG_+$} is a set $\Sigma = \{\Sigma_e : e \in \Gamma_+^1\}$, where each $\Sigma_e$ is a transversal of $G_{r(e)}/\alpha_e(G_e)$ which includes the identity element $1_{r(e)} \in G_{r(e)}$. A $\GG_+$-word $g_1 e_1 g_2 e_2 \ldots g_n e_n g_{n+1}$ is \textit{$\Sigma$-normalised} if $g_i \in \Sigma_{e_i}$ for all $1 \leq i \leq n$.
\end{defin}

\begin{remark}
    Suppose that $\Sigma$ is a set of transversals for a directed graph of groups $\GG_+ = (\Gamma_+, G)$. We can always extend $\Sigma$ to a set of transversals for the underlying graph of groups $\GG$ by simply including transversals $\Sigma_e$ for each edge $e \in \Gamma^1 \setminus \Gamma_+^1$.
\end{remark}

\begin{prop} \label{prop:reduced representative GG+ word}
    Let $\GG_+ = (\Gamma_+, G)$ be a directed graph of groups, and let $\Sigma$ be a set of transversals for $\GG_+$. Every element of $\Lambda_{\GG_+}$ is uniquely represented by a $\Sigma$-normalised $\GG_+$-word.
\end{prop}

\begin{proof}
    Extend $\Sigma$ to a set of transversals for $\GG$. By Proposition~\ref{prop:F(G) Sigma-normalised rep}, every element of $\Lambda_{\GG_+} \subseteq F(\GG)$ is uniquely represented by a $\Sigma$-normalised $\GG$-word. But any such $\GG$-word must only contain edges in orientation $\Gamma_+^1$ (along with group elements), and so the $\GG$-word is a ($\Sigma$-normalised) $\GG_+$-word.
\end{proof}

Using Proposition~\ref{prop:reduced representative GG+ word}, we now relate the word category $\Lambda_{\GG_+}$ to a certain directed graph. We show that this graph encodes the order structure of $\Lambda_{\GG_+}$ (Proposition~\ref{prop:word category relation}), and then we derive some useful consequences.

\begin{notation}
    Let $\GG_+ = (\Gamma_+, G)$ be a directed graph of groups, and let $\Sigma$ be a set of transversals for $\GG_+$. Write $E_\Sigma$ for the directed graph defined by
    \[E_\Sigma^0 = \Gamma^0, \quad E_\Sigma^1 = \{he : e \in \Gamma_+^1,\ h \in \Sigma_e\}, \quad r(he) = r(e), \quad s(he) = s(e).\]
    (The directed graph $E_\Sigma$ is essentially just the directed graph $\Gamma_+$ with its edges replicated a number of times depending on the size of the respective transversal $\Sigma_e$.)
    
    Write $q_\Sigma \colon \Lambda_{\GG_+} \to E_\Sigma^*$ for the map which sends (the image in $\Lambda_{\GG_+}$ of) a $\Sigma$-normalised $\GG_+$-word $h_1 e_1 h_2 e_2 \dots h_n e_n g_{n+1}$ to the path $h_1 e_1 h_2 e_2 \dots h_n e_n$ in $E_\Sigma$ (in the case where $n=0$, we define $q_\Sigma(g_1)$ to be the vertex $v$ for which $g_1 \in G_v$). We will often drop the subscript `$\Sigma$' in the notation and write only `$q$' for this map when $\Sigma$ is understood.
\end{notation}

Note that the path space $E_\Sigma^*$ can be viewed as a left cancellative small category with unit space $E_\Sigma^0 = \Gamma^0$. The relation $\leq$ on $E_\Sigma^*$ as in Definition~\ref{def:LCSC order} is in this case a partial order.

\begin{prop} \label{prop:word category relation}
    Let $\GG_+ = (\Gamma_+, G)$ be a directed graph of groups, and let $\Sigma$ be a set of transversals for $\GG_+$. For any $\lambda, \mu \in \Lambda_{\GG_+}$, we have $\lambda \leq \mu$ in $\Lambda_{\GG_+}$ if and only if $q(\lambda) \leq q(\mu)$ in $E_\Sigma^*$.
\end{prop}

\begin{proof}
    Fix $\lambda, \mu \in \Lambda_{\GG_+}$, and write $\lambda = h_1 e_1 h_2 e_2 \dots h_n e_n g_{n+1}$ in $\Sigma$-normalised form. First suppose that $\lambda \leq \mu$, so $\lambda\nu = \mu$ for some $\nu \in \Lambda_{\GG_+}$. Write $\nu = g'_1 f_1 g'_2 f_2 \dots g'_m f_m g'_{m+1}$, and let $g_{n+1}\nu = h'_1 f_1 h'_2 f_2 \dots h'_m f_m g''_{m+1}$ be $\Sigma$-normalised. Then 
    \[ \mu = \lambda\nu = h_1 e_1 h_2 e_2 \dots h_n e_n h'_1 f_1 h'_2 f_2 \dots h'_m f_m g''_{m+1} \]
    is in $\Sigma$-normalised form, so
    \[q(\lambda) = h_1 e_1 h_2 e_2 \dots h_n e_n \leq h_1 e_1 h_2 e_2 \dots h_n e_n h'_1 f_1 h'_2 f_2 \dots h'_m f_m = q(\mu),\]
    as required.
    
    For the other direction, suppose that $q(\lambda) \leq q(\mu)$, so $\mu$ has $\Sigma$-normalised form
    \[ h_1 e_1 h_2 e_2 \dots h_n e_n h'_1 f_1 h'_2 f_2 \dots h'_m f_m g''_{m+1}. \]
    So $\mu = \lambda (g_{n+1}^{-1} h'_1 f_1 h'_2 f_2 \dots h'_m f_m g''_{m+1})$, showing that $\lambda \leq \mu$, as required.
\end{proof}

\begin{coro} \label{coro:word category extensions}
    Let $\GG_+ = (\Gamma_+, G)$ be a directed graph of groups. For any $\lambda, \mu \in \Lambda_{\GG_+}$, if $\lambda$ and $\mu$ have a common extension, then either $\lambda \leq \mu$ or $\mu \leq \lambda$.
\end{coro}

\begin{proof}
    Suppose that $\lambda$ and $\mu$ have a common extension $\nu$, so $\lambda, \mu \leq \nu$. Then by Proposition~\ref{prop:word category relation}, we have that $q(\lambda), q(\mu) \leq q(\nu)$ in $E_\Sigma^*$; that is, both $q(\lambda)$ and $q(\mu)$ are initial subpaths of $q(\nu)$. But this can only be true if either $q(\lambda)$ is an initial subpath of $q(\mu)$ or vice versa. Then Proposition~\ref{prop:word category relation} implies the claim.
\end{proof}

\begin{coro} \label{coro:Lambda singly aligned}
    Let $\GG_+ = (\Gamma_+, G)$ be a directed graph of groups. The category $\Lambda_{\GG_+}$ is singly aligned.
\end{coro}

\begin{proof}
    Suppose that $\lambda, \mu \in \Lambda_{\GG_+}$ are such that $\lambda \Lambda_{\GG_+} \cap \mu \Lambda_{\GG_+} \neq \emptyset$. Then $\lambda$ and $\mu$ have a common extension, and so by Corollary~\ref{coro:word category extensions} we can assume without loss of generality that $\lambda \leq \mu$. This implies that $\mu \Lambda_{\GG_+} \subseteq \lambda \Lambda_{\GG_+}$ and therefore that $\lambda \Lambda_{\GG_+} \cap \mu \Lambda_{\GG_+} = \mu \Lambda_{\GG_+}$. This proves the claim.
\end{proof}

\subsection{The directed forest $W_{\GG_+}$} \label{subsec:F(G) boundary action}

In this subsection we consider a directed version of the forest $W_\GG$ from Definition~\ref{def:W_G}, and we study the action of the fundamental groupoid $F(\GG)$ on the boundary of this directed forest. This will be useful for Section~\ref{sec:C* Bass-Serre thm}, where we work with the transformation groupoid $F(\GG) \ltimes \partial W_{\GG_+}$. We start by assigning an orientation to $W_\GG$.

\begin{defin}
    Let $\GG_+ = (\Gamma_+, G)$ be a directed graph of groups. Write $W_{\GG_+}$ for the directed forest with $W_\GG$ as its underlying graph, and where an edge from $\gamma' G_w$ to $\gamma G_w$ is in the orientation if and only if $\gamma^{-1} \gamma' \in G_v e G_w$ for some $e \in v\Gamma_+^1 w$.
\end{defin}

Note that for any $x \in \Gamma^0$, the connected component of the vertex $G_x$ in $W_{\GG_+}$ is exactly the directed Bass--Serre tree of $\GG_+$ based at $x$.

We are interested in studying the boundary $\partial W_{\GG_+}$ of $W_{\GG_+}$, and for this it is useful to consider \textit{infinite $\GG$-words}.

\begin{defin}
    Let $\GG = (\Gamma, G)$ be a graph of groups, and let $\Sigma$ be a set of transversals for $\GG$. An \textit{infinite $\GG$-word} is a sequence of the form $g_1 e_1 g_2 e_2 \dots$, where $e_1 e_2 \dots$ is an infinite path in $\Gamma$, and $g_i \in G_{r(e_i)}$ for each $i \geq 1$. An infinite $\GG$-word is \textit{$\Sigma$-normalised} if $g_i \in \Sigma_{e_i}$ for all $i \geq 1$, and there is no subsequence of the form $\overline{e} 1_{r(e)} e$. The range of an infinite $\GG$-word is $r(e_1)$.
\end{defin}

The boundary $\partial X_{\GG, x}$ of the Bass--Serre tree based at $x \in \Gamma^0$ can be identified with the set of $\Sigma$-normalised infinite $\GG$-words with range $x$ (see \cite[Section~2.3.1]{BMPST}), so the boundary $\partial W_\GG$ can be identified with the set of all $\Sigma$-normalised infinite $\GG$-words. The action of $F(\GG)$ on $\partial W_\GG$ induced from the action on $W_\GG$ is as follows. Take $\gamma \in F(\GG)$ and $\xi \in \partial W_\GG$ with $r(\xi) = s(\gamma)$. Then $\gamma$ is represented by some $\GG$-word and $\xi$ is represented by some infinite $\GG$-word, so the concatenation $\gamma\xi$ is also an infinite $\GG$-word. By a (possibly infinite) number of applications of the relations in $F(\GG)$, we can turn $\gamma\xi$ into a $\Sigma$-normalised infinite $\GG$-word, which we call $\gamma \cdot \xi$. It can be verified that this defines an action of $F(\GG)$ on $\partial W_\GG$; we refer to \cite[Section~2]{BMPST} for more detail.

In this paper we are interested in the action of $F(\GG)$ on the boundary $\partial W_{\GG_+}$ of the \textit{directed} forest $W_{\GG_+}$. This boundary can be identified with the set of \textit{eventually directed} $\Sigma$-normalised infinite $\GG$-words, in the following sense.

\begin{defin}
    Let $\GG = (\Gamma_+, G)$ be a directed graph of groups, and let $\Sigma$ be a set of transversals for $\GG$. A $\Sigma$-normalised infinite $\GG$-word $g_1 e_1 g_2 e_2 \dots$ is \textit{eventually directed} if there is some $N \in \N$ such that $e_i \in \Gamma_+^1$ for all $i \geq N$. Equivalently, an eventually directed infinite $\GG$-word has the form $\gamma \alpha$ for some $\gamma \in F(\GG)$ and some \textit{infinite $\GG_+$-word} $\alpha$ (that is, $\alpha$ is an infinite $\GG$-word with all edges belonging to the orientation $\Gamma_+^1$).
\end{defin}

The basic open sets of $\partial W_{\GG_+}$ have the form
\[ Z(\gamma G_{s(\gamma)}) = \{\gamma \alpha : \alpha\ \text{an infinite $\GG_+$-word},\ r(\alpha) = s(\gamma)\} \]
for a fixed $\gamma \in F(\GG)$. Note that if we consider $\partial W_{\GG_+}$ to be a subset of $\partial W_\GG$, then it is invariant under the action of $F(\GG)$, since the concatenation $\gamma\xi$ of a $\GG$-word $\gamma$ and an eventually directed $\Sigma$-normalised infinite $\GG$-word $\xi$ will also be eventually directed. Moreover, the action of $\gamma$ takes a basic open set $Z(\gamma' G_{s(\gamma')})$, $\gamma' \in s(\gamma)F(\GG)$ to the basic open set $Z(\gamma \gamma' G_{s(\gamma')})$ Thus $F(\GG)$ acts on the space $\partial W_{\GG_+}$ by homeomorphisms.

We now focus on the subspace of $\partial W_{\GG_+}$ corresponding to ($\Sigma$-normalised) infinite $\GG_+$-words. These words have the form $h_1 e_1 h_2 e_2 \dots$, where for each $i \leq 1$ we have $e_i \in \Gamma_+^1$ and $h_i \in \Sigma_{e_i}$. But sequences of this form also describe infinite paths in $E_\Sigma^\infty$, and so we can identify $E_\Sigma^\infty$ with a subset of $\partial W_{\GG_+}$. Moreover, the topology of $E_\Sigma^\infty$ matches the topology coming from $\partial W_{\GG_+}$, since the basic open set $\alpha E_\Sigma^\infty \subseteq E_\Sigma^\infty$ corresponds to the basic open set $Z(\alpha 1_{s(\alpha)} G_{s(\alpha)}) \subseteq \partial W_{\GG_+}$; thus we can consider $E_\Sigma^\infty$ to be a subspace of $\partial W_{\GG_+}$. We end this section with the following result which characterises the groupoid elements $\gamma \in F(\GG)$ which send a fixed element of $E_\Sigma^\infty \subseteq \partial W_{\GG_+}$ to another element of $E_\Sigma^\infty$.

\begin{prop} \label{prop:F(G) acting on E_Sigma^infty}
    Let $\GG_+ = (\Gamma_+, G)$ be a directed graph of groups, and let $\Sigma$ be a set of transversals for $\GG_+$. Let $\alpha \in E_\Sigma^\infty$ be an infinite path, considered as an element of $\partial W_\GG$, and let $\gamma \in F(\GG)$. Then $\gamma \cdot \alpha$ is an element of $E_\Sigma^\infty \subseteq \partial W_\GG$ if and only if $\gamma$ has the form $\lambda \mu^{-1}$ for some $\lambda, \mu \in \Lambda_{\GG_+}$, where $q(\mu) \leq \alpha$.
\end{prop}

\begin{proof}
    First suppose that $\gamma = \lambda\mu^{-1}$ for some $\lambda, \mu \in \Lambda_{\GG_+}$ with $q(\mu) \leq \alpha$. Let $g_\mu \in G_{s(\mu)}$ be such that $\mu = q(\mu)g_\mu$, and let $\beta \in E_\Sigma^\infty$ be such that $\alpha = q(\mu)\beta$. Then $\gamma = \lambda g_\mu^{-1} q(\mu)^{-1}$, and so $\gamma \cdot \alpha = \lambda g_\mu^{-1} \beta$ is an infinite $\GG_+$-word, and therefore (its $\Sigma$-normalised version) corresponds to some infinite path in $E_\Sigma$. This proves the backwards direction.
    
    For the forwards direction, extend $\Sigma$ to a set of transversals for $\GG$, and write $\gamma = g_1 e_1 \dots g_n e_n g_{n+1}$ in $\Sigma$-normalised form. If each $e_i \in \Gamma_+^1$, $i=1,\dots,n$, then $\gamma \in \Lambda_{\GG_+}$ and there is nothing to prove, so assume that $e_i \notin \Gamma_+^1$ for some $i$. Write $\alpha = h_1 f_1 h_2 f_2 \dots$. Since $\gamma \cdot \alpha = g_1 e_1 \dots g_n e_n g_{n+1} h_1 f_1 h_2 f_2 \dots$ is directed, the path $e_i$ must be ``canceled out" by the relation $\overline{e}1_{r(e)}e = 1_{s(e)}$ in $F(\GG)$. But since $\gamma$ is already $\Sigma$-normalised, the only possible cancellation that can occur is the sourcemost edge of $\gamma$ with the rangemost edge of $\alpha$; that is, we need $e_n g_{n+1} h_1 f_1$ to have the form $\overline{f_1}1_{r(f_1)}f_1$. This forces $g_{n+1} = h_1^{-1}$ and $e_n = \overline{f_1}$. Write $\gamma' = g_1 e_1 \dots g_{n-1} e_{n-1} g_n$ and $\alpha' = h_2 f_2 h_3 f_3 \dots$, so $\gamma' \cdot \alpha' = \gamma \cdot \alpha$. If each $e_i \in \Gamma_+^1$, $i = 1, \dots, n-1$, then we have that $\gamma' \in \Lambda_{\GG_+}$ and $\gamma = \gamma'(h_1 f_1 1_{s(f_1)})^{-1}$, so we are done. If not, then we repeat the argument with $\gamma'$ and $\alpha'$. As $\gamma$ is a finite $\GG$-word, this process will eventually terminate, at which point we will have that for some $k \leq n$, the word $\gamma'' = g_1 e_1 \dots g_k e_k g_{k+1}$ belongs to $\Lambda_{\GG_+}$, and $\gamma = \gamma''(h_1 f_1 \dots h_{n-k} f_{n-k} 1_{s(f_{n-k})})^{-1}$. Since $q(h_1 f_1 \dots h_{n-k} f_{n-k} 1_{s(f_{n-k})}) = h_1 f_1 \dots h_{n-k} f_{n-k} \leq \alpha$, we have proven the claim.
\end{proof}

\section{Directed graph of groups $C^*$-algebras} \label{sec:dgog algebra}

We now associate a universal $C^*$-algebra to a directed graph of groups, where the generators and relations of the $C^*$-algebra are encoded by the directed graph of groups. Then, we show that this $C^*$-algebra is isomorphic to the Cuntz--Krieger algebra of the word category of the directed graph of groups.

\begin{defin} \label{def:dgog algebra}
    Let $\GG_+ = (\Gamma_+, G)$ be a countable, row-finite directed graph of groups with no sources, and let $\Sigma$ be a set of transversals for $\GG_+$. A \textit{$(\GG_+,\Sigma)$-family} is a collection of partial isometries $S_e$ for each $e \in \Gamma^1_+$ and representations $g \mapsto U_{v,g}$ of $G_v$ by partial unitaries for each $v \in \Gamma^0$ satisfying the relations:
    \begin{enumerate}[(1)]
        \item $U_{v,1} U_{w,1} = 0$ for each $v, w \in \Gamma^0$ with $v \neq w$;
        \item $U_{r(e),\alpha_e(g)} S_e = S_e U_{s(e), \alpha_{\overline{e}}(g)}$ for each $e \in \Gamma^1_+$ and $g \in G_e$;
        \item $U_{s(e),1} = S_e^* S_e$ for each $e \in \Gamma^1_+$; and
        \item $U_{v,1} = \sum_{e \in v\Gamma^1_+, h \in \Sigma_e} U_{v,h}S_e S_e^* U_{v,h}^*$ for each $v \in \Gamma^0$.
    \end{enumerate}
    We write $C^*(\GG_+)$ for the universal $C^*$-algebra generated by a $(\GG_+,\Sigma)$-family, and we call it the \textit{directed graph of groups $C^*$-algebra of $\GG_+$}. We write $\{s_e : e \in \Gamma_+^1\} \cup \{u_{v,g} : v \in \Gamma^0,\ g \in G_v\}$ for the $(\GG_+, \Sigma)$-family generating $C^*(\GG_+)$.
\end{defin}

\textit{A priori}, the $C^*$-algebra $C^*(\GG_+)$ depends on the choice of the set of transversals $\Sigma$. However, it follows from Theorem~\ref{thm:dgogC* isomorphism} below that $C^*(\GG_+)$ is in fact independent (up to isomorphism) of $\Sigma$, which justifies the omission of $\Sigma$ in the notation.

\begin{thm} \label{thm:dgogC* isomorphism}
     Let $\GG_+ = (\Gamma_+, G)$ be a countable, row-finite directed graph of groups with no sources, and let $\Sigma$ be a set of transversals for $\GG_+$. Let $\{t_\lambda : \lambda \in \Lambda_{\GG_+}\}$ be the Cuntz--Krieger $\Lambda_{\GG_+}$-family generating $\OO(\Lambda_{\GG_+})$. The map
    \begin{align*}
        \Phi\colon C^*(\GG_+) &\to \OO(\Lambda_{\GG_+}) \\
        s_e &\mapsto t_{1_{r(e)}e} \\
        u_{v,g} &\mapsto t_g
    \end{align*}
    is an isomorphism of $C^*$-algebras.
\end{thm}

To prove this, we use the universal property of $C^*(\GG)$ and $\OO(\Lambda_{\GG_+})$ to find maps $\Phi\colon C^*(\GG) \to \OO(\Lambda_{\GG_+})$ and $\Psi\colon \OO(\Lambda_{\GG_+}) \to C^*(\GG)$, and then we show that the maps $\Phi$ and $\Psi$ are mutual inverses. We first prove some intermediary results, and then tie everything together in a proof of the theorem at the end of this section.

\begin{lemma} \label{lemma:O(Lambda) partial unitaries}
    Let $\GG_+ = (\Gamma_+, G)$ be a countable, row-finite directed graph of groups with no sources, and let $\{S_\lambda : \lambda \in \Lambda_{\GG_+}\}$ be a family satisfying (S4). For any $v \in \Gamma^0$ and $g \in G_v$, we have that $S_g S_g^* = S_v$.
\end{lemma}

\begin{proof}
    Fix $v \in \Gamma^0$ and $g \in G_v$. The singleton set $\{g\} \subseteq v\Lambda$ is finite exhaustive, since for any $\lambda \in v\Lambda$ we have $g \leq g(g^{-1}\lambda) = \lambda$. Thus (S4) gives that $S_v = S_g S_g^*$, as required.
\end{proof}

\begin{prop} \label{prop:S3'}
    Let $\GG_+ = (\Gamma_+, G)$ be a countable, row-finite directed graph of groups with no sources. A family $\{S_\lambda : \lambda \in \Lambda_{\GG_+}\}$ satisfying (S1), (S2) and (S4) satisfies (S3) if and only if it satisfies the relation
    \begin{enumerate}
        \item[(S3')] $S_\lambda S_\lambda^* S_\mu S_\mu^* = \begin{cases}
            S_\mu S_\mu^* & \text{if}\ \lambda \leq \mu, \\
            S_\lambda S_\lambda^* & \text{if}\ \mu \leq \lambda, \\
            0 & \text{otherwise}.
        \end{cases}$
    \end{enumerate}
\end{prop}

\begin{proof}
    We show that
    \[ \bigvee_{\nu \in \lambda \vee \mu} S_\nu S_\nu^* = 
       \begin{cases}
           S_\mu S_\mu^* & \text{if}\ \lambda \leq \mu, \\
           S_\lambda S_\lambda^* & \text{if}\ \mu \leq \lambda, \\
           0 & \text{otherwise}.
       \end{cases}\]
    Fix $\lambda, \mu \in \Lambda_{\GG_+}$. By Corollary~\ref{coro:word category extensions}, we have that $\lambda \vee \mu \neq \emptyset$ if and only if $\lambda \leq \mu$ or $\mu \leq \lambda$. In the case that $\lambda \leq \mu$, we have that $\lambda \vee \mu$ is the set of $\GG_+$-words $\nu$ such that $\mu \leq \nu$ and $\nu \leq \mu$, which by Proposition~\ref{prop:word category relation} is the set $q^{-1}(\mu) = \{\mu g : g \in G_{s(\mu)}\}$. Now, (S2) and Lemma~\ref{lemma:O(Lambda) partial unitaries} together give that
    \[ S_{\mu g} S_{\mu g}^* = S_\mu S_g S_g^* S_\mu^* = S_\mu S_{s(\mu)} S_\mu^* = S_\mu S_\mu^* \]
    for all $g \in G_{s(\mu)}$. So in this case
    \[ \bigvee_{\nu \in \lambda \vee \mu} S_\nu S_\nu^* = \bigvee_{g \in G_{s(\mu)}} S_{\mu g} S_{\mu g}^* = \bigvee_{g \in G_{s(\mu)}} S_\mu S_\mu^* = S_\mu S_\mu^* \]
    as required. A similar argument in the case where $\mu \leq \lambda$ shows that
    \[ \bigvee_{\nu \in \lambda \vee \mu} S_\nu S_\nu^* = \bigvee_{g \in G_{s(\lambda)}} S_{\lambda g} S_{\lambda g}^* = \bigvee_{g \in G_{s(\lambda)}} S_\lambda S_\lambda^* = S_\lambda S_\lambda^*. \]
    Finally, if neither $\lambda \leq \mu$ nor $\mu \leq \lambda$, then $\lambda \vee \mu = \emptyset$ and so $\bigvee_{\nu \in \lambda \vee \mu} S_\nu S_\nu^* = 0$ as required.
\end{proof}

\begin{prop} \label{prop:dgog family in CK}
    Let $\GG_+ = (\Gamma_+, G)$ be a countable, row-finite directed graph of groups with no sources, and let $\Sigma$ be a set of transversals for $\GG_+$. Let $\{t_\lambda : \lambda \in \Lambda_{\GG_+}\}$ be the Cuntz--Krieger $\Lambda_{\GG_+}$-family generating $\OO(\Lambda_{\GG_+})$. For each $e \in \Gamma_+^1$, let $S_e = t_e$; and for each $v \in \Gamma^0$ and $g \in G_v$, let $U_{v,g} = t_g$. Then $\{S_e : e \in \Gamma_+^1\} \cup \{U_{v,g} : v \in \Gamma^0,\ g \in G_v\}$ is a $(\GG_+, \Sigma)$-family in $\OO(\Lambda_{\GG_+})$.
\end{prop}

\begin{proof}
    Remark~\ref{remark:LCSC family} tells us that $\{S_e : e \in \Gamma_+^1\}$ is a family of partial isometries. Now fix $v \in \Gamma^0$. Each $U_{v,g} = t_g$, $g \in G_v$ is a partial isometry since $t_g^* t_g^* = t_v = t_g t_g^*$ by (S1) and Lemma~\ref{lemma:O(Lambda) partial unitaries}. Moreover, for any $g_1, g_2 \in G_v$ we have $t_{g_1} t_{g_2} = t_{g_1 g_2}$ by (S2), so $G_v \ni g \mapsto U_{v,g}$ is a representation by partial unitaries.
    
    Now we show that the $S_e$'s and $U_{v,g}$'s satisfy the relations in Definition~\ref{def:dgog algebra}. For any $v, w \in \Gamma^0$ with $v \neq w$, (S3') gives that $t_v t_w = t_v t_v^* t_w t_w^* = 0$, so relation (1) is satisfied. For any $e \in \Gamma_+^1$ and $g \in G_e$, we have $t_{\alpha_e(g)} t_e = t_{\alpha_e(g) e} = t_{e \alpha_{\overline{e}}(g)} = t_e t_{\alpha_{\overline{e}}(g)}$ using (S2) and the relation $\alpha_e(g) e = e \alpha_{\overline{e}}(g)$ in the word category, so relation (2) is satisfied. Relation (3) follows immediately from (S1).
    
    Finally, for relation (4), we first claim that for any $v \in \Gamma^0$, the set $\{he : e \in v\Gamma_+^1,\ h \in \Sigma_e\}$ is an exhaustive subset of $v\Lambda_{\GG_+}$. To see this, take $\lambda \in v \Lambda_{\GG_+}$. If $\lambda = g_1$ for some $g_1 \in G_v$, then $\lambda = g_1 \leq g_1(g_1^{-1}he) = he$ for any $e \in v\Gamma_+^1$ and $h \in \Sigma_e$; otherwise if $\lambda = h_1 e_1 \dots h_n e_n g_{n+1}$ in reduced form, then $h_1 e_1 \leq \lambda$. In either case, $\lambda$ is comparable with some element in $\{he : e \in v\Gamma_+^1,\ h \in \Sigma_e\}$, showing that it is exhaustive. Then (S4) gives that
    \[ t_v = \bigvee_{e \in v\Gamma_+^1,\ h \in \Sigma_e} t_{he} t_{he}^*. \]
    But (S3') implies that the $t_{he} t_{he}^*$'s are pairwise orthogonal, and so we get that
    \[ t_v = \sum_{e \in v\Gamma_+^1,\ h \in \Sigma_e} t_{he} t_{he}^* = \sum_{e \in v\Gamma_+^1,\ h \in \Sigma_e} t_h t_e t_e^* t_h^*, \]
    which is relation (4). This proves that $\{S_e : e \in \Gamma_+^1\} \cup \{U_{v,g} : v \in \Gamma^0,\ g \in G_v\}$ is a $(\GG_+, \Sigma)$-family in $\OO(\Lambda_{\GG_+})$, as claimed.
\end{proof}

\begin{prop} \label{prop:CK family in dgog}
    Let $\GG_+ = (\Gamma_+, G)$ be a countable, row-finite directed graph of groups with no sources, and let $\Sigma$ be a set of transversals for $\GG_+$. For each $\lambda = g_1 e_1 g_2 e_2 \dots g_n e_n g_{n+1}$ in $\Lambda_{\GG_+}$, define
    \[ S_\lambda = u_{r(e_1),g_1} s_{e_1} u_{r(e_2),g_2} s_{e_2} \dots u_{r(e_n),g_n} s_{e_n} u_{s(e_n),g_{n+1}}. \]
    Then $\{S_\lambda : \lambda \in \Lambda_{\GG_+}\}$ is a Cuntz--Krieger $\Lambda_{\GG_+}$-family in $C^*(\GG_+)$.
\end{prop}

\begin{proof}
    First note that the $S_\lambda$'s are well-defined, since if the $\GG_+$-words $g_1 e_1 g_2 e_2 \dots g_n e_n g_{n+1}$ and $h_1 e_1 h_2 e_2 \dots h_n e_n h_{n+1}$ are two presentations of the same element $\lambda \in \Lambda_{\GG_+} \subseteq F(\GG)$, then they are related via a finite number of applications of the relation $\alpha_e(g) e = e \alpha_{\overline{e}}(g)$ which is implemented by relation (2) of Definition~\ref{def:dgog algebra}.
    
    To see that (S1) holds, take $\lambda = g_1 e_1 g_2 e_2 \dots g_n e_n g_{n+1}$. Note that we have
    \[u_{r(e_i),g_i}^* u_{r(e_i),g_i} = u_{r(e_i),1} = s_{e_i}^* s_{e_i}\]
    for $i=1,\dots,n$, and also $u_{s(e_n),g_{n+1}}^* u_{s(e_n),g_{n+1}} = u_{s(e),1}$. Applying these relations iteratively to $S_\lambda^* S_\lambda$ gives that $S_\lambda^* S_\lambda = u_{s(e_n),1} = t_{s(\lambda)} = S_{s(\lambda)}$, as required.
    
    (S2) holds by definition of the $S_\lambda$'s. For (S3), Proposition~\ref{prop:S3'} means that it is enough to show that (S3') holds. Take $\lambda, \mu \in \Lambda_{\GG_+}$. We have four cases. If $\lambda \leq \mu$, then there exists $\nu \in \Lambda_{\GG_+}$ with $\mu = \lambda\nu$. So
    \[ S_\lambda S_\lambda^* S_\mu S_\mu^* = S_\lambda S_\lambda^* S_\lambda S_\nu S_\nu^* S_\lambda^* = S_\lambda S_\nu S_\nu^* S_\lambda^* = S_\mu S_\mu^*. \]
    Similarly if $\mu \leq \lambda$, then there exists $\nu \in \Lambda_{\GG_+}$ with $\lambda = \mu\nu$. So
    \[ S_\lambda S_\lambda^* S_\mu S_\mu^* = S_\mu S_\nu S_\nu^* S_\mu^* S_\mu S_\mu^* = S_\mu S_\nu S_\nu^* S_\mu^* = S_\lambda S_\lambda^*. \]
    If $r(\lambda) \neq r(\mu)$, then $\lambda$ and $\mu$ are not comparable, and relation (1) gives that
    \[ S_\lambda S_\lambda^* S_\mu S_\mu^* = S_\lambda S_\lambda^* u_{r(\lambda),1_{r(\lambda)}} u_{r(\mu),1_{r(\mu)}} S_\mu S_\mu^* = 0. \]
    Finally we come to the case where $\lambda$ and $\mu$ have the same range but are not comparable. Write $\nu$ for the maximal common initial subpath of $\lambda$ and $\mu$. Then
    \[ \lambda = \nu(h_1 e_1)\lambda' \quad \text{and} \quad \mu = \nu(h_2 e_2)\mu' \]
    for some $h_1 e_1, \lambda', h_2 e_2, \mu' \in \Lambda_{\GG_+}$ with $h_1 e_1 \neq h_2 e_2$. So
    \[ S_\lambda^* S_\mu = S_{\lambda'}^* S_{h_1 e_1}^* S_\nu^* S_\nu S_{h_2 e_2} S_{\mu'} =  S_{\lambda'}^* S_{h_1 e_1}^* S_{h_2 e_2} S_{\mu'} = 0 \]
    since $S_{h_1 e_1}$ and $S_{h_2 e_2}$ have orthogonal range projections due to relation (4).
    
    For (S4), fix $v \in \Gamma^0$ and a finite exhaustive set $F \subseteq v\Lambda_{\GG_+}$. First note that relation (4) says that $u_{w,1} = \sum_{he \in wE_\Sigma^1} S_{he} u_{s(e),1} S_{he}^*$ for all $w \in \Gamma^0$, and applying this inductively gives that $u_{v,1} = \sum_{\alpha \in E_\Sigma^n} S_\alpha S_\alpha^*$ for all $n \in \N$. Now, for any $\lambda \in F$, setting $n = \abs{\lambda}$ (that is, the number of edges in the $\GG_+$-word $\lambda$) shows that $S_\lambda S_\lambda^* = S_{q(\lambda)} S_{q(\lambda)}^* \leq u_{v,1}$, so $\bigvee_{\lambda \in F} S_\lambda S_\lambda^* \leq u_{v,1}$.
    
    It remains to prove that $u_{v,1} \leq \bigvee_{\lambda \in F} S_\lambda S_\lambda^*$. Let $n = \max_{\lambda \in F} \abs{q(\lambda)}$. Since $F$ is exhaustive, for each $\alpha \in E_\Sigma^n$ there is some $\lambda_\alpha \in F$ such that $\alpha 1_{s(\alpha)}$ and $\lambda_\alpha$ have a common extension. But Corollary~\ref{coro:word category extensions} and the choice of length of $\alpha$ force $\lambda_\alpha \leq \alpha 1_{s(\alpha)}$. This implies that $S_{\alpha 1_{s(\alpha)}} S_{\alpha 1_{s(\alpha)}}^* \leq S_{\lambda_\alpha} S_{\lambda_\alpha}^* \leq \bigvee_{\lambda \in F} S_\lambda S_\lambda^*$ for all $\alpha \in E_\Sigma^n$, and so
    \[ u_{v,1} = \sum_{\alpha \in E_\Sigma^n} S_\alpha S_\alpha^* = \bigvee_{\alpha \in E_\Sigma^n} S_{\alpha 1_{s(\alpha)}} S_{\alpha 1_{s(\alpha)}}^* \leq \bigvee_{\lambda \in F} S_\lambda S_\lambda^*, \]
    as required.
\end{proof}

We now put everything together to prove Theorem~\ref{thm:dgogC* isomorphism}.

\begin{proof}[Proof of Theorem~\ref{thm:dgogC* isomorphism}]
    Proposition~\ref{prop:dgog family in CK} along with the universal property of $C^*(\GG_+)$ show that the map $\Phi$ in the statement of the theorem is a well-defined $*$-homomorphism. Proposition~\ref{prop:CK family in dgog} along with the universal property of $\OO(\Lambda_{\GG_+})$ give a $*$-homomorphism $\Psi\colon \OO(\Lambda_{\GG_+}) \to C^*(\GG_+)$ satisfying $\Psi(t_e) = s_e$ for $e \in \Gamma_+^1$ and $\Psi(t_g) = u_{v,g}$ for $v \in \Gamma^0$ and $g \in G_v$. So $\Psi \circ \Phi$ is the identity map on $C^*(\GG_+)$, which means that $\Phi$ is injective. To see that $\Phi$ is surjective, note that any $\lambda \in \Lambda_{\GG_+}$ is a finite composition of elements of the form $e \in \Gamma_+^1$ and $g \in G_v$, $v \in \Gamma^0$. So the elements $t_e = \Phi(s_e)$ and $t_g = \Phi(u_{v,g})$ generate $\OO(\Lambda_{\GG_+})$, meaning that $\Phi$ is surjective. Hence $\Phi$ is a $C^*$-isomorphism, as claimed.
\end{proof}

\section{A $C^*$-algebraic `directed' Bass--Serre theorem} \label{sec:C* Bass-Serre thm}

This section is dedicated to proving the following result, which can be seen as the directed version of \cite[Theorem~4.1]{BMPST}, described in that paper as a $C^*$-algebraic Bass--Serre theorem. This result also generalises \cite[Corollary~4.14]{Kumjian-Pask}, which gives the result in the setting of directed graph algebras.

\begin{thm} \label{thm:morita equiv}
    Let $\GG_+ = (\Gamma_+, G)$ be a countable, row-finite directed graph of groups with no sources, and choose a base vertex $x \in \Gamma^0$. Then $C^*(\GG_+)$ is Morita equivalent to $C_0(\partial X_{\GG_+,x}) \rtimes \pi_1(\GG,x)$.
\end{thm}

To prove this result, we first describe a groupoid model for $C^*(\GG_+)$ in Section~\ref{subsec:C*(G+) groupoid model}, and then in Section~\ref{subsec:groupoid equivalence} we show that the underlying groupoid is equivalent in a suitable sense to the transformation groupoid $\pi_1(\GG, x) \ltimes \partial X_{\GG_+,x}$.

\subsection{Groupoid model for $C^*(\GG_+)$} \label{subsec:C*(G+) groupoid model}

Ortega and Pardo \cite{LCSCsemigroup} give a groupoid model for the Cuntz--Krieger algebra of a countable, finitely aligned LCSC using the language of inverse semigroup actions. In this subsection we analyse the model as it applies to our context, and provide a tractable version of the model in Theorem~\ref{thm:semigroup crossed product}. We begin with a general discussion of the groupoid model from \cite{LCSCsemigroup}.

Let $\Lambda$ be a countable, finitely aligned LCSC. We can associate a semigroup $\SS_\Lambda$ to $\Lambda$ in the following way.

\begin{defin}[\mbox{\cite[Section~2]{LCSCsemigroup}}]
    Let $\Lambda$ be a LCSC. For each $\lambda \in \Lambda$, we define maps
    \[\tau^\lambda \colon s(\lambda)\Lambda \to \lambda\Lambda; \quad \mu \mapsto \lambda\mu, \qquad \text{and} \qquad \sigma^\lambda \colon \lambda\Lambda \to s(\lambda)\Lambda; \quad \lambda\mu \mapsto \mu.\]
    We write $\SS_\Lambda$ for the semigroup generated by $\{\tau^\lambda, \sigma^\lambda : \lambda \in \Lambda\}$ under composition, where a composition is deemed to equal $0$ if it has empty domain.
\end{defin}

The semigroup $\SS_\Lambda$ is an inverse semigroup, called the left inverse hull of $\Lambda$. Theorem~4.16 of \cite{LCSCsemigroup} states that the Cuntz--Krieger algebra of $\Lambda$ is isomorphic to the $C^*$-algebra of what is called the \textit{tight groupoid} of $\SS_\Lambda$. This groupoid is topologically isomorphic to the groupoid of germs of the action of $\SS_\Lambda$ on a certain space $\Lambda_{tight}$ \cite[Lemma~4.9]{LCSCsemigroup}, so $\OO(\Lambda) \cong C^*(\SS_\Lambda \ltimes \Lambda_{tight})$.

We will now describe the space $\Lambda_{tight}$ and the action of  $\SS_\Lambda$ on $\Lambda_{tight}$.

\begin{defin}[\mbox{\cite[Section~3]{LCSCsemigroup}}]
    Let $\Lambda$ be a finitely aligned LCSC. A non-empty subset $H \subseteq \Lambda$ is:
    \begin{enumerate}[(1)]
        \item \textit{hereditary} if for any $\lambda \in H$ and $\mu \in \Lambda$ such that $\mu \leq \lambda$, we have that $\mu \in H$;
        \item \textit{(upwards) directed} if for any $\lambda, \mu \in H$, there exists some $\nu \in H$ such that $\lambda, \mu \leq \nu$; and
        \item \textit{tight} if for any $\lambda \in H$, $\{\mu_1, \dots, \mu_n\} \subseteq \Lambda \setminus H$, and finite exhaustive set $Z$ of $\lambda\Lambda \setminus \bigcup_{i=1}^n \mu_i \Lambda$, we have that $Z \cap H \neq \emptyset$.
    \end{enumerate}
    Write $\Lambda^*$ for the collection of hereditary directed subsets of $\Lambda$, and write $\Lambda_{tight}$ for the collection of tight hereditary directed subsets of $\Lambda$.
\end{defin}

We can turn $\Lambda^*$ and $\Lambda_{tight}$ into topological spaces as follows. For finite sets $X,Y \subseteq \Lambda$, define
\[\MM^{X,Y}= \{H \in \Lambda^* : X \subseteq H\ \text{and}\ Y \cap H = \emptyset \}.\]
We can endow $\Lambda^*$ with a topology generated by open sets $\{\MM^{X,Y} : X,Y \subseteq \Lambda\ \text{finite}\}$, and we give $\Lambda_{tight}$ the subspace topology.

The semigroup $\SS_\Lambda$ acts on $\Lambda^*$ as follows. The domain and range of a generating element $\tau^\lambda \sigma^\mu \in \SS_\Lambda$ (where $\lambda, \mu \in \Lambda$ with $s(\lambda) = s(\mu)$) are $\{H \in \Lambda^* : \mu \in H\}$ and $\{H \in \Lambda^* : \lambda \in H\}$ respectively, and $\tau^\lambda \sigma^\mu$ acts according to the formula
\[ \tau^\lambda \sigma^\mu \cdot H = \{\xi \in \Lambda : \xi \leq  \tau^\lambda \sigma^\mu(\nu)\ \text{for some}\ \nu \in H\ \text{such that}\ \mu \leq \nu\} \]
for $H$ in the domain of $\tau^\lambda \sigma^\mu$. This action restricts to an action on $\Lambda_{tight}$ \cite[Corollary~4.7]{LCSCsemigroup}, whose groupoid of germs has $C^*$-algebra isomorphic to $\OO(\Lambda)$.

We now consider these results in the setting where the category $\Lambda$ is the word category of a directed graph of groups. To avoid cluttered notation, we will write $\SS_{\GG_+}$ instead of $\SS_{\Lambda_{\GG_+}}$ to denote the left inverse hull of the word category of a directed graph of groups $\GG_+$. The following proposition describes the structure of $\SS_{\GG_+}$.

\begin{prop} \label{prop:S_G+ structure}
    Let $\GG_+$ be a directed graph of groups. We have that
    \[\SS_{\GG_+} = \{\tau^\lambda \sigma^\mu : \lambda, \mu \in \Lambda_{\GG_+},\ s(\lambda) = s(\mu)\} \cup \{0\}.\]
    Moreover, for any two nonzero elements $\tau^\lambda \sigma^\mu, \tau^\nu \sigma^\xi \in \SS_{\GG_+}$, we have
    \begin{enumerate}
        \item $\tau^\lambda \sigma^\mu = \tau^\nu \sigma^\xi$ if and only if there is $g \in G_{s(\lambda)}$ such that $\nu = \lambda g$ and $\xi = \mu g$; and
        
        \item
        \[ \tau^\lambda \sigma^\mu \tau^\nu \sigma^\xi =
        \begin{cases}
            \tau^\lambda \sigma^{\xi\eta} & \text{if}\ \mu = \nu \eta\ \text{for some}\ \eta \in \Lambda_{\GG_+}; \\
            \tau^{\lambda\eta} \sigma^\xi & \text{if}\ \nu = \mu \eta\ \text{for some}\ \eta \in \Lambda_{\GG_+}; \\
            0 & \text{otherwise.}
        \end{cases} \]
    \end{enumerate}
\end{prop}

\begin{proof}
    Since $\Lambda_{\GG_+}$ is singly aligned (Corollary~\ref{coro:Lambda singly aligned}), \cite[Theorem~3.2]{singly_aligned} gives that
    \[\SS_{\GG_+} = \{\tau^\lambda \sigma^\mu : \lambda, \mu \in \Lambda_{\GG_+},\ s(\lambda) = s(\mu)\} \cup \{0\},\]
    and that multiplication is described by
    \[ \tau^\lambda \sigma^\mu \tau^\nu \sigma^\xi =
    \begin{cases}
        \tau^{\lambda\sigma^\mu(\gamma)} \sigma^{\xi\sigma^\nu(\gamma)} & \text{if}\ \mu \Lambda_{\GG_+} \cap \nu \Lambda_{\GG_+} = \gamma \Lambda_{\GG_+}; \\
        0 & \text{otherwise.}
    \end{cases} \]
    (Note that in \cite{singly_aligned}, the semigroup $\SS_{\GG_+}$ is denoted as $ZM(\Lambda_{\GG_+})$.) Now, the condition that $\mu \Lambda_{\GG_+} \cap \nu \Lambda_{\GG_+}$ is non-empty is equivalent to $\mu$ and $\nu$ having a common extension, which by Corollary~\ref{coro:word category extensions} is in turn equivalent to $\mu$ and $\nu$ being comparable. In this case, we can consider two (overlapping) subcases: $\nu \leq \mu$ or $\mu \leq \nu$.
    
    If $\nu \leq \mu$, then $\mu = \nu\eta$ for some $\eta \in \Lambda_{\GG_+}$, and we can take $\gamma$ to be $\mu$. Then $\sigma^\mu(\gamma)$ is just $s(\mu)$ and $\sigma^\nu(\gamma) = \eta$, so we get that $\tau^{\lambda\sigma^\mu(\gamma)} \sigma^{\xi\sigma^\nu(\gamma)} = \tau^\lambda \sigma^{\xi\eta}$. If instead $\mu \leq \nu$, then $\nu = \mu\eta$ for some $\eta \in \Lambda_{\GG_+}$, and we can take $\gamma$ to be $\nu$. Then $\sigma^\nu(\gamma)$ is just $s(\nu)$ and $\sigma^\mu(\gamma) = \eta$, so we get that $\tau^{\lambda\sigma^\mu(\gamma)} \sigma^{\xi\sigma^\nu(\gamma)} = \tau^{\lambda\eta} \sigma^\xi$. This proves (2).
    
    Finally, (1) follows from \cite[Proposition~3.3]{singly_aligned}, since invertible elements in $\Lambda_{\GG_+}$ at the object $s(\lambda)$ are exactly elements of the vertex group $G_{s(\lambda)}$.
\end{proof}

\begin{remark}
    If all groups in $\GG_+$ are trivial, then $\SS_{\GG_+}$ is the graph inverse semigroup of $\Gamma_+$ as defined in \cite{graph_semigroup}.
\end{remark}

We now turn to the space $(\Lambda_{\GG_+})_{tight}$ of tight hereditary directed subsets of $\Lambda_{\GG_+}$. The aim of the next series of results is to prove that $(\Lambda_{\GG_+})_{tight}$ is isomorphic to the infinite path space $E_\Sigma^\infty$ of $E_\Sigma$, for any set of transversals $\Sigma$. We prove this in several steps: first we show that $(\Lambda_{\GG_+})_{tight}$ is in bijection with $(E_\Sigma^*)_{tight}$ (Proposition~\ref{prop:Lambda E* tight}); next we show that $(E_\Sigma^*)_{tight}$ is in turn in bijection with $E_\Sigma^\infty$ (Proposition~\ref{prop:E* tight infinite paths}); and finally we show that the bijection between $(\Lambda_{\GG_+})_{tight}$ and $E_\Sigma^\infty$ is a homeomorphism (Proposition~\ref{prop:Lambda tight E infty}).

\begin{prop} \label{prop:Lambda E* tight}
    Let $\GG_+$ be a directed graph of groups, and let $\Sigma$ be a set of transversals for $\GG_+$. The map $q\colon \Lambda_{\GG_+} \to E_\Sigma^*$ induces a bijection $(\Lambda_{\GG_+})_{tight} \to (E_\Sigma^*)_{tight}$, whose inverse is given by taking pre-images under $q$.
\end{prop}

\begin{proof}
    First note that for any hereditary subset $H \subseteq \Lambda_{\GG_+}$, if $\lambda \in H$ then, by Proposition~\ref{prop:word category relation}, $H$ also contains all $\mu \in \Lambda_{\GG_+}$ for which $q(\mu) = q(\lambda)$; thus $H = q^{-1}(q(H))$. Moreover, $q(H)$ is hereditary, since for any $\alpha \in E_\Sigma^*$ and $\mu \in H$ such that $\alpha \leq q(\mu)$, Proposition~\ref{prop:word category relation} gives that $\alpha \leq \mu$ in $\Lambda_{\GG_+}$ and thus that $\alpha \in H$ since $H$ is hereditary; so we have $\alpha = q(\alpha) \in q(H)$ as required. Likewise, if $K \subseteq E_\Sigma^*$ is hereditary, then its pre-image $q^{-1}(K) \subseteq \Lambda_{\GG_+}$ is also hereditary. So $q$ induces a bijection between the hereditary subsets of $\Lambda_{\GG_+}$ and $E_\Sigma^*$, with the inverse given by taking pre-images under $q$. Now, Proposition~\ref{prop:word category relation} also implies that a subset $H \subseteq \Lambda_{\GG_+}$ is directed if and only if its image $q(H) \subseteq E_\Sigma^*$ is directed, so $q$ in fact induces a bijection between $\Lambda_{\GG_+}^*$ and $(E_\Sigma^*)^*$.
    
    We now show that this bijection restricts to a bijection on the level of tight hereditary directed subsets. First suppose that $H \subseteq \Lambda_{\GG_+}$ is a tight hereditary directed subset; we show that $q(H) \subseteq E_\Sigma^*$ is tight. Take $\alpha \in q(H)$ and $\beta_1, \dots, \beta_n \in E_\Sigma^* \setminus q(H)$, and let $Z \subseteq \alpha E_\Sigma^* \setminus \bigcup_{i=1}^n \beta_i E_\Sigma^*$ be finite exhaustive. Since $H = q^{-1}(q(H))$, if we consider $\alpha, \beta_1, \dots, \beta_n$ as elements of $\Lambda_{\GG_+}$, then $\alpha \in H$ and $\beta_1, \dots, \beta_n \in \Lambda_{\GG_+} \setminus H$. If we also identify $Z$ with its image in $\Lambda_{\GG_+}$, then $Z \subseteq \alpha \Lambda_{\GG_+} \setminus \bigcup_{i=1}^n \beta_i \Lambda_{\GG_+}$ is also finite exhaustive: for any $\lambda \in \alpha \Lambda_{\GG_+} \setminus \bigcup_{i=1}^n \beta_i \Lambda_{\GG_+}$, we have $q(\lambda) \in \alpha E_\Sigma^* \setminus \bigcup_{i=1}^n \beta_i E_\Sigma^*$, so since $Z$ is exhaustive there is some $z \in Z$ such that $q(\lambda)E_\Sigma^* \cap z E_\Sigma^* \neq \emptyset$; but then Proposition~\ref{prop:word category relation} implies that $\lambda \Lambda_{\GG_+} \cup z \Lambda_{\GG_+} \neq \emptyset$ also. Then, since $H$ is tight, we have $Z \cap H \neq \emptyset$, and so $Z \cap q(H) \neq \emptyset$ in $E_\Sigma^*$. Hence $q(H)$ is tight. A similar argument shows that if $K \subseteq E_\Sigma^*$ is a tight hereditary directed subset, then its pre-image $q^{-1}(K) \subseteq \Lambda_{\GG_+}$ is also tight. Hence $q$ induces a bijection $(\Lambda_{\GG_+})_{tight} \to (E_\Sigma^*)_{tight}$, as claimed.
\end{proof}

\begin{prop} \label{prop:E* tight infinite paths}
    Let $\GG_+$ be a countable, row-finite directed graph of groups with no sources, and let $\Sigma$ be a set of transversals for $\GG_+$. For each infinite path $\alpha \in E_\Sigma^\infty$, define a subset $\Delta_\alpha \subseteq E_\Sigma^*$ by $\Delta_\alpha := \{\beta \in E_\Sigma^* : \beta \leq \alpha\}$. Write $\Delta$ for the map sending $\alpha$ to $\Delta_\alpha$. Then $\Delta$ is a bijection $E_\Sigma^\infty \to (E_\Sigma^*)_{tight}$.
\end{prop}

\begin{proof}
    It is easy to check that each $\Delta_\alpha$, $\alpha \in E_\Sigma^\infty$ is hereditary and directed. To see that it is tight, fix paths $\beta \in \Delta_\alpha$ and $\gamma_1, \dots, \gamma_n \in E_\Sigma^* \setminus \Delta_\alpha$, and fix a finite exhaustive subset $Z \subseteq \beta E_\Sigma^* \setminus \bigcup_{i=1}^n \gamma_i E_\Sigma^*$; we will show that $Z \cap \Delta_\alpha \neq \emptyset$. Take $\beta' \in \Delta_\alpha$ such that $\beta \leq \beta'$ and $\abs{\beta'} > \max_{z \in Z} \abs{z}$. Note that $\gamma_i \not\leq \beta'$ for $i=1, \dots, n$ (otherwise, since $\Delta_\alpha$ is hereditary we would have $\gamma_i \in \Delta_\alpha$, contradicting the choice of $\gamma_i$), so $\beta' \in \beta E_\Sigma^* \setminus \bigcup_{i=1}^n \gamma_i E_\Sigma^*$. Then, since $Z$ is exhaustive, there is some $z \in Z$ which has a common extension with $\beta'$. This implies that either $z \leq \beta'$ or $\beta' \leq z$. But the latter is impossible, since we chose $\beta'$ to be longer than any path in $Z$; so we must have $z \leq \beta'$. Finally, since $\Delta_\alpha$ is hereditary, we get that $z \in \Delta_\alpha$, which implies that $Z \cap \Delta_\alpha \neq \emptyset$, as required. Hence each $\Delta_\alpha$, $\alpha \in E_\Sigma^\infty$ is tight, and so $\Delta$ is a map $E_\Sigma^\infty \to (E_\Sigma^*)_{tight}$.
    
    Since every infinite path in $E_\Sigma$ is uniquely determined by the set of its initial subpaths, we get that $\Delta$ is injective. Now, the discussion in \cite[Example~3.31]{LCSCsemigroup} (where our $\Delta_\alpha$ is denoted as $E_\alpha$) shows that every tight hereditary directed subset of $E_\Sigma^*$ is equal to $\Delta_\alpha$ for some $\alpha \in E_\Sigma^\infty$ (note that since $\GG_+$ is row-finite and has no sources, the same is true for $E_\Sigma$, and so $E_\Sigma$ has no singular vertices), so $\Delta$ is also surjective. Thus $\Delta\colon E_\Sigma^\infty \to (E_\Sigma^*)_{tight}$ is a bijection, as claimed.
\end{proof}

\begin{prop} \label{prop:Lambda tight E infty}
    Let $\Delta\colon E_\Sigma^\infty \to (E_\Sigma^*)_{tight}$ be as in Proposition~\ref{prop:E* tight infinite paths}. The map $q^{-1} \circ \Delta \colon E_\Sigma^\infty \to (\Lambda_{\GG_+})_{tight}$ is a homeomorphism.
\end{prop}

\begin{proof}
    Propositions~\ref{prop:Lambda E* tight} and \ref{prop:E* tight infinite paths} together give that $q^{-1} \circ \Delta$ is a bijection, so it remains to prove that it is continuous and open. Recall that the codomain $(\Lambda_{\GG_+})_{tight}$ has basic open sets of the form
    \[ \MM^{X,Y} = \{H \in (\Lambda_{\GG_+})_{tight} : X \subseteq H\ \text{and}\ Y \cap H = \emptyset\}, \quad X, Y \subseteq \Lambda_{\GG_+}\ \text{finite}. \]
    Fix finite sets $X, Y \subseteq \Lambda_{\GG_+}$. For an infinite path $\alpha \in E_\Sigma^\infty$, the condition $X \subseteq q^{-1}(\Delta_\alpha)$ means that $q(\lambda) \leq \alpha$ for all $\lambda \in X$, and similarly the condition $Y \cap q^{-1}(\Delta_\alpha) = \emptyset$ means that $q(\mu) \not\leq \alpha$ for all $\mu \in Y$. So
    \begin{align*}
        \Delta^{-1}(q(\MM^{X,Y})) &= \{\alpha \in E_\Sigma^\infty : q(\lambda) \leq \alpha\ \text{for all}\ \lambda \in X,\ \text{and}\ q(\mu) \not\leq \alpha\ \text{for all}\ \mu \in Y\} \\
        &= \left(\bigcap_{\lambda \in X} q(\lambda) E_\Sigma^\infty\right) \setminus \left(\bigcup_{\mu \in Y} q(\mu) E_\Sigma^\infty\right),
    \end{align*}
    which is a basic open set in $E_\Sigma^\infty$, and moreover, all basic open sets in $E_\Sigma^\infty$ have this form. This implies that $q^{-1} \circ \Delta$ is continuous and open, and hence is a homeomorphism.
\end{proof}

We are now ready to prove the main result of this subsection, which gives a more tangible groupoid model for a directed graph of groups $C^*$-algebra.

\begin{thm} \label{thm:semigroup crossed product}
    Let $\GG_+$ be a countable, row-finite directed graph of groups with no sources, and let $\Sigma$ be a set of transversals for $\GG_+$. Then $C^*(\GG_+) \cong C^*(\SS_{\GG_+} \ltimes E_\Sigma^\infty)$, where $\SS_{\GG_+} \ltimes E_\Sigma^\infty$ is the groupoid of germs for the action of $\SS_{\GG_+}$ on $E_\Sigma^\infty$ given by
    \[ \tau^\lambda \sigma^\mu \cdot \alpha = \lambda \mu^{-1} \cdot \alpha, \quad \text{for}\ \alpha\ \text{such that}\ q(\mu) \leq \alpha, \]
    where in the right hand side the action of $\lambda \mu^{-1} \in F(\GG)$ on $\alpha \in E_\Sigma^\infty$ is as in Proposition~\ref{prop:F(G) acting on E_Sigma^infty}.
\end{thm}

\begin{proof}
    From Theorem~\ref{thm:dgogC* isomorphism} we know that $C^*(\GG_+) \cong \OO(\Lambda_{\GG_+})$, and \cite[Section~4.4]{LCSCsemigroup} gives that $\OO(\Lambda_{\GG_+}) \cong C^*(\SS_{\GG_+} \ltimes (\Lambda_{\GG_+})_{tight})$. Proposition~\ref{prop:Lambda tight E infty} says that the map $q^{-1} \circ \Delta \colon E_\Sigma^\infty \to (\Lambda_{\GG_+})_{tight}$ is a homeomorphism, and so it remains to show that the induced action of $\SS_{\GG_+}$ on $E_\Sigma^\infty$ has the formula given in the theorem statement (this would also imply that the formula defines a legitimate action).
    
    Recall that the action of $\SS_{\GG_+}$ on $(\Lambda_{\GG_+})_{tight}$ is given by
    \[ \tau^\lambda \sigma^\mu \cdot H = \{\xi \in \Lambda_{\GG_+} : \xi \leq \tau^\lambda \sigma^\mu(\nu)\ \text{for some}\ \nu \in H\ \text{such that}\ \mu \leq \nu\} \]
    for $H \in (\Lambda_{\GG_+})_{tight}$ such that $\mu \in H$. Pulling back through the homeomorphism $q^{-1}\circ \Delta$, we get that for $\alpha \in E_\Sigma^\infty$ such that $q(\mu) \leq \alpha$,
    \begin{align*}
        \tau^\lambda \sigma^\mu \cdot q^{-1}(\Delta_\alpha) &= \{\xi \in \Lambda_{\GG_+} : \xi \leq \tau^\lambda \sigma^\mu(\nu)\ \text{for some}\ \nu \in q^{-1}(\Delta_\alpha)\ \text{such that}\ \mu \leq \nu\} \\
        &= \{\xi \in \Lambda_{\GG_+} : \xi \leq \tau^\lambda \sigma^\mu(\nu)\ \text{for some}\ \nu \in \Lambda_{\GG_+}\ \text{such that}\ q(\mu) \leq q(\nu) \leq \alpha\}.
    \end{align*}
    But note that the condition that $\xi \leq \tau^\lambda \sigma^\mu(\nu)$ for some $\nu \in \Lambda_{\GG_+}$ with $q(\mu) \leq q(\nu) \leq \alpha$ is equivalent to the condition that $q(\xi) \leq \lambda\mu^{-1} \cdot \alpha$, so $\tau^\lambda \sigma^\mu \cdot q^{-1}(\Delta_\alpha) = q^{-1}(\Delta_{\lambda\mu^{-1}\cdot \alpha})$. Thus the action of $\SS_{\GG_+}$ on $E_\Sigma^\infty$ induced from the action on $(\Lambda_{\GG_+})_{tight}$ is given by $\tau^\lambda \sigma^\mu \cdot \alpha = \lambda\mu^{-1} \cdot \alpha$, as claimed.
\end{proof}

To finish this subsection, we study the groupoid of germs $\SS_{\GG_+} \ltimes E_\Sigma^\infty$ from Theorem~\ref{thm:semigroup crossed product}. Recall that the groupoid $\SS_{\GG_+} \ltimes E_\Sigma^\infty$ consists of equivalence classes of elements in $\SS_{\GG_+} * E_\Sigma^\infty = \{(\tau^\lambda \sigma^\mu, \alpha) \in \SS_{\GG_+} \times E_\Sigma^\infty : s(\lambda) = s(\mu)\ \text{and}\ q(\mu) \leq \alpha\}$. We now give a concrete description of the equivalence in our context.

\begin{prop} \label{prop:germ equality}
    Let $(\tau^\lambda \sigma^\mu, \alpha), (\tau^{\lambda'} \sigma^{\mu'}, \beta) \in \SS_{\GG_+} * E_\Sigma^\infty$. The germs $[\tau^\lambda \sigma^\mu, \alpha]$ and $[\tau^{\lambda'} \sigma^{\mu'}, \beta]$ are the same if and only if $\alpha = \beta$ and (assuming without loss of generality that $\mu \leq \mu'$) there exists $\eta \in \Lambda_{\GG_+}$ such that $\lambda' = \lambda \eta$ and $\mu' = \mu \eta$.
\end{prop}

\begin{proof}
    By definition, we have that the germs $[\tau^\lambda \sigma^\mu, \alpha]$ and $[\tau^{\lambda'} \sigma^{\mu'}, \beta]$ are the same if and only if $\alpha = \beta$ and there is an idempotent $e \in \SS_{\GG_+}$ such that $\alpha \in D_e$ and $\tau^\lambda \sigma^\mu e = \tau^{\lambda'} \sigma^{\mu'} e$. Note that $e$ must be non-zero, since otherwise $D_e$ would be empty. Now, \cite[Proposition~3.5(1)]{singly_aligned} gives that the non-zero idempotents of $\SS_{\GG_+}$ are exactly elements of the form $\tau^\xi \sigma^\xi$ for some $\xi \in \Lambda_{\GG_+}$, and we have that $\alpha \in D_{\tau^\xi \sigma^\xi}$ if and only if $q(\xi) \leq \alpha$. So the germs $[\tau^\lambda \sigma^\mu, \alpha]$ and $[\tau^{\lambda'} \sigma^{\mu'}, \beta]$ are the same if and only if $\alpha = \beta$ and there is $\xi \in \Lambda_{\GG_+}$ such that $q(\xi) \leq \alpha$ and $\tau^\lambda \sigma^\mu \tau^\xi \sigma^\xi = \tau^{\lambda'} \sigma^{\mu'} \tau^\xi \sigma^\xi$.
    
    First suppose that there is $\eta \in \Lambda_{\GG_+}$ such that $\lambda' = \lambda \eta$ and $\mu' = \mu \eta$. Then Proposition~\ref{prop:S_G+ structure} gives that $\tau^\lambda \sigma^\mu \tau^{\mu'}\sigma^{\mu'} = \tau^{\lambda'} \sigma^{\mu'} = \tau^{\lambda'} \sigma^{\mu'} \tau^{\mu'}\sigma^{\mu'}$, so we can take $\xi = \mu'$ to see that the germs $[\tau^\lambda \sigma^\mu, \alpha]$ and $[\tau^{\lambda'} \sigma^{\mu'}, \beta]$ are the same.
    
    Now suppose that the germs $[\tau^\lambda \sigma^\mu, \alpha]$ and $[\tau^{\lambda'} \sigma^{\mu'}, \beta]$ are the same. Note that since $\alpha = \beta$, we must have both $q(\mu) \leq \alpha$ and $q(\mu') \leq \alpha$. This forces $q(\mu)$ and $q(\mu')$ to be comparable, since they are both initial subpaths of $\alpha$. Assume without loss of generality that $q(\mu) \leq q(\mu')$; then Proposition~\ref{prop:word category relation} gives that $\mu \leq \mu'$, so there is some $\eta \in \Lambda_{\GG_+}$ such that $\mu' = \mu \eta$.
    
    It is left to show that $\lambda' = \lambda \eta$. Let $\xi \in \Lambda_{\GG_+}$ be such that $q(\xi) \leq \alpha$ and $\tau^\lambda \sigma^\mu \tau^\xi \sigma^\xi = \tau^{\lambda'} \sigma^{\mu'} \tau^\xi \sigma^\xi$. Since $q(\xi)$, $q(\mu)$ and $q(\mu')$ are all subpaths of $\alpha$, we have that $q(\xi)$ is comparable with both $q(\mu)$ and $q(\mu')$. Hence by Proposition~\ref{prop:word category relation} we have that $\xi$ is comparable with both $\mu$ and $\mu'$.
    
    We now reduce to the case where $(\mu \leq)\ \mu' \leq \xi$. Indeed, suppose that $\xi \leq \mu'$, so $\mu' = \xi\nu$ for some $\nu \in \Lambda_{\GG_+}$. Then $\sigma^\xi(\mu') = \nu$, and $\tau^\xi \sigma^\xi \tau^{\mu'} \sigma^{\mu'} = \tau^{\xi\nu} \sigma^{\mu'} = \tau^{\mu'} \sigma^{\mu'}$. This means that
    \begin{align*}
        \tau^\lambda \sigma^\mu \tau^{\mu'} \sigma^{\mu'} &= \tau^\lambda \sigma^\mu \tau^\xi \sigma^\xi \tau^{\mu'} \sigma^{\mu'} \\
        &= \tau^{\lambda'} \sigma^{\mu'} \tau^\xi \sigma^\xi \tau^{\mu'} \sigma^{\mu'} \\
        &= \tau^{\lambda'} \sigma^{\mu'} \tau^{\mu'} \sigma^{\mu'}.
    \end{align*}
    We also have that $\alpha \in D_{\tau^{\mu'} \sigma^{\mu'}}$, so $\tau^{\mu'} \sigma^{\mu'}$ is also an idempotent in $\SS_{\GG_+}$ with the required properties. Hence we can replace $\xi$ with $\mu'$, so that now $\mu' \leq \xi$.
    
    Finally we prove the statement in the case that $(\mu \leq)\ \mu' \leq \xi$. Let $\nu, \nu' \in \Lambda_{\GG_+}$ be such that $\xi = \mu \nu = \mu' \nu'$. Then $\tau^\lambda \sigma^\mu \tau^\xi \sigma^\xi = \tau^{\lambda\nu} \sigma^\xi$, and $\tau^{\lambda'} \sigma^{\mu'} \tau^\xi \sigma^\xi = \tau^{\lambda'\nu'} \sigma^\xi$. So the assumption on $\xi$ implies that $\tau^{\lambda\nu} \sigma^\xi = \tau^{\lambda'\nu'} \sigma^\xi$, and then \cite[Proposition~3.3]{singly_aligned} gives that $\lambda\nu = \lambda'\nu'$. Now, $\mu\nu = \xi = \mu'\nu' = \mu \eta \nu'$, so by cancellativity of $\Lambda_{\GG_+}$ we have that $\nu = \eta \nu'$. So $\lambda' \nu' = \lambda \nu = \lambda \eta \nu'$, and using cancellativity again gives that $\lambda' = \lambda \eta$, as required.
\end{proof}

\subsection{Groupoid equivalence} \label{subsec:groupoid equivalence}

We now use the groupoid model for $C^*(\GG_+)$ obtained in Theorem~\ref{thm:semigroup crossed product} to prove Theorem~\ref{thm:morita equiv}, that $C^*(\GG_+)$ is Morita equivalent to the crossed product $C_0(\partial X_{\GG_+, x}) \rtimes \pi_1(\GG, x)$. Our strategy is to show that the underlying groupoids of the two $C^*$-algebras in question, namely the groupoid of germs $\SS_{\GG_+} \ltimes E_\Sigma^\infty$ and the transformation groupoid $\pi_1(\GG, x) \ltimes \partial X_{\GG_+,x}$, are equivalent in a suitable sense so that we have a Morita equivalence on the level of $C^*$-algebras.

We begin with a discussion of groupoid equivalence. In \cite{Renault}, Renault introduced the notion of groupoid equivalence: roughly speaking, two Hausdorff \'etale groupoids $\mathscr{G}$ and $\mathscr{H}$ are equivalent if they act in a compatible way on some space. For us, the important property is that equivalent groupoids have Morita equivalent $C^*$-algebras (\cite[Theorem~2.8]{Renault_groupoid}).

Rather than working with this notion of groupoid equivalence directly, we will consider the notion of \textit{weak Kakutani equivalence} for groupoids, which was studied in \cite{ample_groupoids}. Recall that a subset $X$ of the unit space of a groupoid $\mathscr{G}$ is called \textit{full} if it meets the orbit of every unit in $\mathscr{G}$. Two Hausdorff \'etale groupoids are called \textit{weakly Kakutani equivalent} if there are full open subsets $U \subseteq \mathscr{G}^{(0)}$ and $V \subseteq \mathscr{H}^{(0)}$ such that $\mathscr{G}|_U \cong \mathscr{H}|_V$ \cite[Definition~3.8]{ample_groupoids}. Weak Kakutani equivalence implies equivalence in the sense of Renault \cite[Proposition~3.10]{ample_groupoids}.

We combine these concepts in the following key lemma, which we will use to prove Theorem~\ref{thm:morita equiv}.

\begin{lemma} \label{lemma:restrictions Morita equiv}
    Let $\mathscr{G}$ be a Hausdorff \'etale groupoid, and suppose that $U, V \subseteq \mathscr{G}^{(0)}$ are full open subsets. Then $C^*(\mathscr{G}|_U)$ and $C^*(\mathscr{G}|_V)$ are Morita equivalent $C^*$-algebras.
\end{lemma}

\begin{proof}
    Note that $\mathscr{G}|_U$ is weakly Kakutani equivalent to $\mathscr{G}$ (take $U$ to be the full open subset for both groupoids), and similarly $\mathscr{G}|_V$ is also weakly Kakutani equialent to $\mathscr{G}$. Then \cite[Proposition~3.10]{ample_groupoids} implies that $\mathscr{G}|_U$ and $\mathscr{G}|_V$ are both equivalent to $\mathscr{G}$ in the sense of Renault. Since equivalence in the sense of Renault is an equivalence relation (see e.g.\ \cite[p.\ 6]{Renault_groupoid}), we get that $\mathscr{G}|_U$ and $\mathscr{G}|_V$ are equivalent, and hence their $C^*$-algebras are Morita equivalent by \cite[Theorem~2.8]{Renault_groupoid}.
\end{proof}

\begin{prop} \label{prop:groupoid of germs is a restriction}
    Let $\GG_+$ be a countable, row-finite directed graph of groups with no sources, and let $\Sigma$ be a set of transversals for $\GG_+$. The groupoid of germs $\SS_{\GG_+} \ltimes E_\Sigma^\infty$ is topologically isomorphic to the groupoid $(F(\GG) \ltimes \partial W_{\GG_+}) |_{E_\Sigma^\infty}$.
\end{prop}

\begin{proof}
    Propositions~\ref{prop:germ equality} and \ref{prop:F(G) acting on E_Sigma^infty} together mean that we can define a map
    \[ \Theta \colon \SS_{\GG_+} \ltimes E_\Sigma^\infty \to (F(\GG) \ltimes \partial W_{\GG_+}) |_{E_\Sigma^\infty}, \quad [\tau^\lambda \sigma^\mu, \alpha] \mapsto (\lambda\mu^{-1}, \alpha), \]
    and Theorem~\ref{thm:semigroup crossed product} implies that $\Theta$ is a groupoid homomorphism.
    We claim that $\Theta$ is a topological groupoid isomorphism.
    
    Proposition~\ref{prop:F(G) acting on E_Sigma^infty} gives that $\Theta$ is surjective. For injectivity, suppose that $[\tau^\lambda \sigma^\mu, \alpha]$ and $[\tau^\nu \sigma^\xi, \beta]$ are two germs in $\SS_{\GG_+} \ltimes E_\Sigma^\infty$ such that $\Theta([\tau^\lambda \sigma^\mu, \alpha]) = \Theta([\tau^\nu \sigma^\xi, \beta])$. So $(\lambda\mu^{-1}, \alpha) = (\nu\xi^{-1}, \beta)$, which means that $\alpha = \beta$ and $\lambda\mu^{-1} = \nu\xi^{-1}$. Now, since $\alpha$ is in the domain of both $\tau^\lambda \sigma^\mu$ and $\tau^\nu \sigma^\xi$, we must have that $q(\mu)$ and $q(\xi)$ are both initial subpaths of $\alpha$ and therefore must be comparable. Without loss of generality, suppose that $q(\mu) \leq q(\xi)$. Then Proposition~\ref{prop:word category relation} gives that $\mu \leq \xi$, so there is some $\eta \in \Lambda_{\GG_+}$ such that $\xi = \mu\eta$. Then $\lambda\mu^{-1} = \nu(\mu\eta)^{-1} = \nu\eta^{-1}\mu^{-1}$ in $F(\GG)$, so $\lambda = \nu\eta^{-1}$ and thus $\nu = \lambda \eta$. Then Proposition~\ref{prop:germ equality} gives that $[\tau^\lambda \sigma^\mu, \alpha] = [\tau^\nu \sigma^\xi, \beta]$, proving that $\Theta$ is injective.
    
    Finally we show that $\Theta$ is a homeomorphism. Take a basic open set $[\tau^\lambda \sigma^\mu, U] \subseteq \SS_{\GG_+} \ltimes E_\Sigma^\infty$, where $U \subseteq q(\mu)E_\Sigma^\infty$ is open. The image of this set under $\Theta$ is $(\lambda \mu^{-1}, U)$, which is open in $(F(\GG) \ltimes \partial W_{\GG_+}) |_{E_\Sigma^\infty}$. Similarly, the inverse image of the basic open set $(\lambda \mu^{-1}, U) \subseteq (F(\GG) \ltimes \partial W_{\GG_+}) |_{E_\Sigma^\infty}$ is $[\tau^\lambda \sigma^\mu, U]$, which is open in $\SS_{\GG_+} \ltimes E_\Sigma^\infty$. Thus $\Theta$ induces a bijection between basic open sets, and is therefore a homeomorphism. This proves that the groupoids $\SS_{\GG_+} \ltimes E_\Sigma^\infty$ and $(F(\GG) \ltimes \partial W_{\GG_+}) |_{E_\Sigma^\infty}$ are topologically isomorphic.
\end{proof}

We are now ready to prove Theorem~\ref{thm:morita equiv}.

\begin{proof}[Proof of Theorem~\ref{thm:morita equiv}]
    Let $\Sigma$ be a set of transversals for $\GG_+$. Proposition~\ref{prop:groupoid of germs is a restriction} says that the groupoid of germs $\SS_{\GG_+} \ltimes E_\Sigma^\infty$ is topologically isomorphic to the groupoid $(F(\GG) \ltimes \partial W_{\GG_+}) |_{E_\Sigma^\infty}$, and by definition we have that the transformation groupoid $\pi_1(\GG, x) \ltimes \partial X_{\GG_+,x}$ is the groupoid $(F(\GG) \ltimes \partial W_{\GG_+}) |_{\partial X_{\GG_+,x}}$. We claim that both $E_\Sigma^\infty$ and $\partial X_{\GG_+,x}$ are full open subsets of the unit space of $F(\GG) \ltimes \partial W_{\GG_+}$.
    
    First we consider $E_\Sigma^\infty$. This set is open as $E_\Sigma^\infty = \bigsqcup_{v \in \Gamma^0} Z(G_v)$. To see that it is full, take any $\gamma \in \partial W_{\GG_+}$, viewed as an eventually directed infinite $\GG$-word. Write $\gamma = \gamma' \alpha$, where $\alpha$ is an infinite $\GG_+$-word (which we can view as an element of $E_\Sigma^\infty$, as in Section~\ref{subsec:F(G) boundary action}). Then $\gamma$ is in the same orbit as $\alpha \in E_\Sigma^\infty$, showing that $E_\Sigma^\infty$ is full.
    
    Now consider $\partial X_{\GG_+,x}$. This set is open as $\partial X_{\GG_+,x} = \bigcup_{\gamma \in x\GG^*} Z(\gamma G_{s(\gamma)})$. To see that it is full, again take any $\gamma \in \partial W_{\GG_+}$, viewed as an eventually directed infinite $\GG$-word. Write $v$ for the range of $\gamma$, and let $\gamma_v \in x F(\GG) v$. Then $\gamma_v \gamma$ is an eventually directed infinite $\GG$-word with range $x$, and so can be viewed as an element of $\partial X_{\GG_+, x}$. This shows that $\partial X_{\GG_+, x}$ is full.
    
    Hence we can apply Lemma~\ref{lemma:restrictions Morita equiv}, which together with Proposition~\ref{prop:groupoid of germs is a restriction} gives that $C^*(\SS_{\GG_+} \ltimes E_\Sigma^\infty)$ is Morita equivalent to $C^*(\pi_1(\GG, x) \ltimes \partial X_{\GG_+,x})$. The theorem follows, since $C^*(\GG_+) \cong C^*(\SS_{\GG_+} \ltimes E_\Sigma^\infty)$ by Theorem~\ref{thm:semigroup crossed product}, and $C_0(\partial X_{\GG_+, x}) \rtimes \pi_1(\GG, x) \cong C^*(\pi_1(\GG, x) \ltimes \partial X_{\GG_+,x})$.
\end{proof}

\begin{remark} \label{rmk:stable isomorphism}
    We can in fact say a little more about the relationship between $C^*(\GG_+)$ and $C_0(\partial X_{\GG_+, x}) \rtimes \pi_1(\GG, x)$, under the assumptions of Theorem~\ref{thm:morita equiv}. Both $C^*$-algebras are separable since they have countable generating sets, and two separable $C^*$-algebras are Morita equivalent if and only if they are stably isomorphic \cite{stable_isomorphism_morita_equivalence}. Thus we have that $C^*(\GG_+)$ is stably isomorphic to $C_0(\partial X_{\GG_+, x}) \rtimes \pi_1(\GG, x)$.
\end{remark}

\section{A construction of all stable UCT Kirchberg algebras} \label{sec:Kirchberg}

In this final section we aim to show that the class of directed graph of groups $C^*$-algebras includes all stable UCT Kirchberg algebras (recall that a Kirchberg algebra is a $C^*$-algebra that is simple, separable, purely infinite, and nuclear). These algebras are the subject of the celebrated Kirchberg--Phillips classification theorem \cite{Kirchberg, Phillips}, part of which we state below.

\begin{thm}[Kirchberg--Phillips classification theorem]
    Suppose that $A$ and $B$ are both stable UCT Kirchberg algebras. Then $A \cong B$ if and only if $K_*(A) \cong K_*(B)$.
\end{thm}

One of the benefits of such a classification result is that it allows us realise these abstractly defined $C^*$-algebras via a concrete construction, as has been done in \cite{Spielberg_hybrid_graphs} and \cite{Katsura}. Our main result of this section, which we state below, adds to this list of constructions; notably, our result gives the first concrete construction of Kirchberg algebras as crossed product $C^*$-algebras.

\begin{thm} \label{thm:dgog Kirchberg}
    Let $A$ be a stable UCT Kirchberg algebra. There exists a countable, row-finite directed graph of infinite cyclic groups $\GG_+ = (\Gamma_+, G)$ with no sources, such that $A \cong C^*(\GG_+) \cong C_0(\partial X_{\GG_+, x}) \rtimes \pi_1(\GG_+, x)$ for any $x \in \Gamma^0$.
\end{thm}

In order to show this, we use Theorem~\ref{thm:morita equiv} along with the results in \cite{BSTW}, which provides both a $K$-theory formula for crossed products coming from group actions on directed tree boundaries as well as criteria for such a crossed product to be a UCT Kirchberg algebra. We then transfer these results to our context of directed graph of groups $C^*$-algebras using the observation that the class of UCT Kirchberg algebras is closed under stable isomorphism (see \cite{Rordam}).

\subsection{$C^*$-algebras of directed graphs of infinite cyclic groups}

In this subsection, we investigate the $C^*$-algebras of directed graphs of groups where all vertex and edge groups are infinite cyclic (we will call these \textit{directed graphs of infinite cyclic groups}). In particular, we give a formula for the $K$-theory of these $C^*$-algebras (Proposition~\ref{prop:Ktheory}), as well as sufficient conditions for the $C^*$-algebra to be a UCT Kirchberg algebra (Proposition~\ref{prop:Kirchberg conditions}).

We start with some notation.

\begin{notation} \label{not:Z-graph n_e m_e}
    Let $\GG_+ = (\Gamma_+, G)$ be a directed graph of infinite cyclic groups. We fix generators $1_v \in G_v$ and $1_e \in G_e$ for each vertex $v \in \Gamma^0$ and each edge $e \in \Gamma_+^1$ so that for each edge $e \in \Gamma_+^1$, the map $\alpha_e\colon G_e \to G_{r(e)}$ is given by multiplication by some \textit{positive} integer $n_e \in \N \setminus \{0\}$ (this is always possible to arrange); and the map $\alpha_{\overline{e}} \colon G_e \to G_{s(e)}$ will be given by multiplication by some non-zero integer $m_e \in \Z \setminus \{0\}$.
\end{notation}

\begin{remark} \label{rmk:infinite cyclic BSTW translation}
    By Theorem~\ref{thm:directed Bass-Serre}, we have that the vertex stabilisers of the action $\pi_1(\GG,x) \curvearrowright X_{\GG_+,x}$ are infinite cyclic, and that the quotient directed graph of the action is isomorphic to $\Gamma_+$. Moreover, for each $e \in \Gamma_+^1$, the integers $n_e$ and $m_e$ from Notation~\ref{not:Z-graph n_e m_e} are exactly the integers $\omega_e$ and $\omega_{\overline{e}}$ of \cite[Notation~4.18]{BSTW} respectively.
\end{remark}

We can now give a formula for the $K$-theory of $C^*(\GG_+)$ for a directed graph of infinite cyclic groups $\GG_+$.

\begin{prop} \label{prop:Ktheory}
    Let $\GG_+ = (\Gamma_+, G)$ be a countable, row-finite directed graph of infinite cyclic groups with no sources. Let $N, M \colon \Z[\Gamma^0] \to \Z[\Gamma^0]$ be the $\Z$-module homomorphisms given by
    \[ N(v) = \sum_{e\in v\Gamma^1_+} n_e s(e) \qquad \text{and} \qquad M(v) = \sum_{e\in v\Gamma^1_+} m_e s(e) \qquad \text{for}\ v \in \Gamma^0. \]
    Then
    \begin{align*}
        K_0(C^*(\GG_+)) &\cong \coker(1-N) \oplus \ker(1-M), \quad \text{and} \\
        K_1(C^*(\GG_+)) &\cong \coker(1-M) \oplus \ker(1-N).
    \end{align*}
\end{prop}

\begin{proof}
    We know from Theorem~\ref{thm:morita equiv} that $C^*(\GG_+)$ is Morita equivalent to the crossed product $C_0(\partial X_{\GG_+, x}) \rtimes \pi_1(\GG, x)$. Since $K$-theory is invariant under Morita equivalence, it is enough to show the corresponding isomorphisms for the $K$-theory of the crossed product. To do this, we apply \cite[Theorem~4.19]{BSTW}. The tree $X_{\GG_+, x}$ is row-finite and has no sources since $\Gamma_+$ has those properties, and it is a finitely aligned multitree in the sense of \cite{BSTW} since it is a tree. The group $\pi_1(\GG, x)$ is countable since $\GG_+$ is countable, and it acts amenably on $\partial X_{\GG_+, x}$ by \cite[Proposition~6.11]{BSTW} (as per Remark~\ref{rmk:infinite cyclic BSTW translation}, all vertex stabilisers of the action are infinite cyclic, and hence amenable). So the assumptions of \cite[Theorem~4.19]{BSTW} are satisfied, and we get that
    \begin{align*}
        K_0(C_0(\partial X_{\GG_+, x}) \rtimes_r \pi_1(\GG, x)) &\cong \coker(1-A_0) \oplus \ker(1-A_1), \quad \text{and} \\
        K_1(C_0(\partial X_{\GG_+, x}) \rtimes_r \pi_1(\GG, x)) &\cong \coker(1-A_1) \oplus \ker(1-A_0),
    \end{align*}
    where $A_0$ and $A_1$ are the $\Z$-module homomorphisms $\Z[\Gamma^0] \to \Z[\Gamma^0]$ defined by
    \[ A_0(v) = \sum_{e \in v\Gamma_+^1} \abs{\omega_e} s(e) \quad \text{and} \quad A_1(v) = \sum_{e \in v\Gamma_+^1} \operatorname{sgn}(\omega_e)\omega_{\overline{e}} s(e) \]
    for $v \in \Gamma^0$. But by Remark~\ref{rmk:infinite cyclic BSTW translation}, we have that $\omega_e = n_e$ and $\omega_{\overline{e}} = m_e$ for all $e \in \Gamma_+^1$, and recall that each $n_e$ was chosen to be positive. Thus the maps $A_0$ and $A_1$ are exactly the maps $N$ and $M$ from the proposition statement, respectively. Finally, we note that the reduced crossed product $C_0(\partial X_{\GG_+, x}) \rtimes_r \pi_1(\GG, x)$ is the same as the full crossed product $C_0(\partial X_{\GG_+, x}) \rtimes \pi_1(\GG, x)$ since the action of $\pi_1(\GG, x)$ on $\partial X_{\GG_+, x}$ is amenable. This concludes the proof.
\end{proof}

Next, we give sufficient conditions for $C^*(\GG_+)$ to be a UCT Kirchberg algebra. In the following, for $q \in \Q^\times$ we denote by $\langle q \rangle$ the smallest positive integer that can appear as a denominator when expressing $q$ as a fraction. 

\begin{prop} \label{prop:Kirchberg conditions}
    Let $\GG_+ = (\Gamma_+, G)$ be a countable, row-finite directed graph of infinite cyclic groups with no sources. Suppose that:
    \begin{enumerate}
        \item the directed graph $\Gamma_+$ is cofinal;
        \item there is a loop $\eta = \eta_1\eta_2\dots\eta_m$ in $\Gamma_+$ such that either $\eta$ has an entrance, or $n_{\eta_i} \geq 2$ for some $1 \leq i \leq m$; and
        \item there is some infinite path $\alpha = e_1 e_2 \ldots \in \Gamma_+^\infty$ such that
        \[ \limsup_{k \to \infty} \left\langle \frac{m_{e_1}\dots m_{e_{k-1}}}{n_{e_1}\dots n_{e_k}} \right\rangle = \infty. \]
    \end{enumerate}
    Then $C^*(\GG_+)$ is a UCT Kirchberg algebra.
\end{prop}

\begin{proof}
    Since the class of UCT Kirchberg algebras is closed under stable isomorphism, Remark~\ref{rmk:stable isomorphism} means that it is enough to show that the crossed product $C_0(\partial X_{\GG_+,x}) \rtimes \pi_1(\GG,x)$ is a UCT Kirchberg algebra.
    
    Separability is automatic, since $\pi_1(\GG,x)$ is countable. Note that the action $\pi_1(\GG,x) \curvearrowright \partial X_{\GG_+,x}$ is amenable by \cite[Proposition~6.11]{BSTW}, so the crossed product is nuclear \cite{amenable_groupoids_book} and satisfies the UCT (see \cite{Tu_UCT}). So it remains to show that the crossed product is simple and purely infinite.
    
    By \cite[p.124]{topological_freeness}, the crossed product will be simple if the action $\pi_1(\GG,x) \curvearrowright \partial X_{\GG_+,x}$ is minimal and topologically free (we already know the action to be amenable); and in this case the crossed product will be purely infinite if the action is locally contractive \cite[Theorem 9]{purely_infinite_LS} (where local contractivity is called being a \textit{local boundary action}). By Remark~\ref{rmk:infinite cyclic BSTW translation}, conditions (1) and (3) of the proposition are exactly the conditions in \cite[Propositions~6.1,~6.6]{BSTW} which are equivalent to the action $\pi_1(\GG,x) \curvearrowright \partial X_{\GG_+,x}$ being minimal and topologically free, respectively.
    
    Finally, we need to show that the action $\pi_1(\GG,x) \curvearrowright \partial X_{\GG_+,x}$ is locally contractive. Take the loop $\eta$ from condition (2). Since $n_e = [G_{r(e)}:\alpha(G_e)]$ for any $e \in \Gamma_+^1$, we get that $\eta$ satisfies condition (i) of \cite[Proposition~6.2]{BSTW}; moreover, the confinality of $\Gamma_+$ means that there is a path from $r(\eta)$ to any vertex $v \in \Gamma^0$ (consider the infinite path $\eta \eta \dots$), so condition (ii) is also satisfied. Hence \cite[Proposition~6.2]{BSTW} gives that $\pi_1(\GG,x) \curvearrowright \partial X_{\GG_+,x}$ is locally contractive, as required.
\end{proof}

\subsection{Construction of stable UCT Kirchberg algebras}

We now describe how to construct a directed graph of (infinite cyclic) groups whose $C^*$-algebra is a stable UCT Kirchberg algebra with any given $K$-theory; together with Theorem~\ref{thm:morita equiv}, this essentially proves Theorem~\ref{thm:dgog Kirchberg}. Our construction is inspired by the approach in \cite{Katsura}, but is adjusted slightly to work in our context. We begin by explaining how we can define a directed graph of infinite cyclic groups from a pair of integer matrices.

\begin{notation} \label{notation:G^N,M}
    Let $J$ be a countable index set, and let $N \in M_{\abs{J}}(\N)$ and $M \in M_{\abs{J}}(\Z)$ be two matrices indexed over $J$ satisfying $N_{i,j} = 0 \iff M_{i,j} = 0$ for all $i,j \in J$. Define $\GG_+^{N,M} = (\Gamma_+, G)$ to be the directed graph of groups where:
    \begin{itemize}
        \item the underlying graph $\Gamma_+ = (\Gamma^0, \Gamma_+^1, r, s)$ is given by
        \[\Gamma^0 = J, \qquad \Gamma_+^1 = \{(i,j) \in J \times J : N_{i,j} > 0\}, \qquad r(i,j) = i, \qquad s(i,j) = j;\]
        \item all vertex and edge groups are infinite cyclic, with a choice of generator $1_i \in G_i$ for each $i \in \Gamma^0$ and $1_{(i,j)} \in G_{(i,j)}$ for each $(i,j) \in \Gamma_+^1$; and
        \item for each $(i,j) \in \Gamma_+^1$, the monomorphisms $\alpha_{(i,j)}$ and $\alpha_{\overline{(i,j)}}$ are defined by
        \[\alpha_{(i,j)}(1_{(i,j)}) = N_{i,j} 1_i, \qquad \alpha_{\overline{(i,j)}}(1_{(i,j)}) = M_{i,j} 1_j. \]
    \end{itemize}
    We also have the following canonical set of tranvsersals for $\GG^{N,M}_+$: for each $(i,j) \in \Gamma_+^1$, we let $\Sigma_{(i,j)} = \{0, 1_i, 2\times 1_i, \dots, (N_{i,j}-1)1_i\}$.
\end{notation}

\begin{remark} \label{rmk:dgog construction Ktheory}
    Let $N, M$ and $\GG_+^{N,M}$ be as in Notation~\ref{notation:G^N,M}. The maps $N, M \colon \Z[\Gamma^0] \to \Z[\Gamma^0]$ of Proposition~\ref{prop:Ktheory} for $\GG_+^{N,M}$ are precisely the maps given by left multiplication by the matrices $N$ and $M$. Thus we have
    \begin{align*}
        K_0(C^*(\GG_+^{N,M})) &\cong \coker(1-N) \oplus \ker(1-M), \quad \text{and} \\
        K_1(C^*(\GG_+^{N,M})) &\cong \coker(1-M) \oplus \ker(1-N).
    \end{align*}
\end{remark}

In order to prove Theorem~\ref{thm:dgog Kirchberg} we show that for any pair of countable abelian groups $G_0, G_1$ we can find suitable matrices $N, M$ such that $C^*(\GG_+^{N,M})$ is a UCT Kirchberg algebra with $K$-theory $(G_0, G_1)$.

\begin{construction} \label{construction:N,M}
    Let $G_0$ and $G_1$ be countable abelian groups. There exist injective homomorphisms $T, S \colon \Z^\infty \to \Z^\infty$ such that $\coker T \cong G_0$ and $\coker S\cong G_1$. We view these as elements of $M_\infty(\Z)$, indexed over $\N$, and we will write $\abs{T}$ and $\abs{S}$ for the matrices obtained by taking element-wise absolute values of $T$ and $S$ respectively. Define $X \in M_\infty(\N)$ by
    \[X_{k,l} = \begin{cases}
        1 & \text{if}\ \abs{k-l} = 1, \\
        0 & \text{otherwise},
    \end{cases}\]
    for $k,l \in \N$, and define $Y \in M_\infty(\N)$ by $Y = 2\abs{T} + \abs{S} + X$, both indexed over $\N$. Define matrices $N \in M_\infty(\N)$ and $M \in M_\infty(\Z)$ in block form by
    \[N = \begin{pmatrix} 2I & T + Y \\ I & I + Y \end{pmatrix} \quad \text{and} \quad M = \begin{pmatrix} 3I & S + 2Y \\ I & I + Y \end{pmatrix}.\]
    We index $N$ and $M$ over $\N \times \{0,1\}$, where the first block of rows/columns corresponds to the indices $\N \times \{0\}$, and the second block corresponds to the indices $\N \times \{1\}$.
\end{construction}

\begin{prop} \label{prop:dgog construction N,M}
    The matrices $N, M$ of Construction~\ref{construction:N,M} satisfy:
    \begin{enumerate}
        \item $N_{i,j} = 0 \iff M_{i,j} = 0$ for all $i,j \in \N \times \{0,1\}$;
        \item the maps $1-N$ and $1-M$ are injective; and
        \item $\coker(1-N) \cong G_0$ and $\coker(1-M) \cong G_1$.
    \end{enumerate}
\end{prop}

\begin{proof}
    To prove (1), it is enough to check that $(T+Y)_{k,l} = 0 \iff (S+2Y)_{k,l} = 0$ for all $k,l \in \N$. First note that for any integer $a \in \Z$, the number $a + 2\abs{a}$ is always non-negative and is equal to $0$ if and only if $a=0$. Thus we have that, for all $k,l \in \N$, both $(T+Y)_{k,l} = (T_{k,l} + 2\abs{T_{k,l}}) + \abs{S_{k,l}} + X_{k,l}$ and $(S+2Y)_{k,l} = (S_{k,l} + 2\abs{S_{k,l}}) + 4\abs{T_{k,l}} + 2X_{k,l}$ equal $0$ if and only if $T_{k,l} = S_{k,l} = X_{k,l} = 0$, proving the claim.
    
    For (2) and (3), note that
    \[ 1-N = \begin{pmatrix} -I & -T-Y \\ -I & -Y \end{pmatrix} = \begin{pmatrix} I & I \\ 0 & I \end{pmatrix}\begin{pmatrix} 0 & -T \\ -I & 0 \end{pmatrix}\begin{pmatrix} I & Y \\ 0 & I \end{pmatrix} \]
    and
    \[ 1-M = \begin{pmatrix} -2I & -S-2Y \\ -I & -Y \end{pmatrix} = \begin{pmatrix} I & 2I \\ 0 & I \end{pmatrix}\begin{pmatrix} 0 & -S \\ -I & 0 \end{pmatrix}\begin{pmatrix} I & Y \\ 0 & I \end{pmatrix}. \]
    
    Since matrices of the form $\begin{psmallmatrix} I & * \\ 0 & I \end{psmallmatrix}$ are invertible (with inverse $\begin{psmallmatrix} I & -* \\ 0 & I \end{psmallmatrix}$), we get that
    \begin{align*}
        \ker(1-N) &= \ker \begin{pmatrix} 0 & -T \\ -I & 0 \end{pmatrix} = 0; \\
        \coker(1-N) &= \coker \begin{pmatrix} 0 & -T \\ -I & 0 \end{pmatrix} = \coker T = G_0; \\
        \ker(1-M) &= \ker \begin{pmatrix} 0 & -S \\ -I & 0 \end{pmatrix} = 0;\ \text{and} \\
        \coker(1-M) &= \coker \begin{pmatrix} 0 & -S \\ -I & 0 \end{pmatrix} = \coker S = G_1,
    \end{align*}
    as claimed.
\end{proof}

Remark~\ref{rmk:dgog construction Ktheory} and Proposition~\ref{prop:dgog construction N,M} together give the following result.

\begin{coro} \label{coro:C*(G^N,M) K-theory}
    Let $N$ and $M$ be as in Construction~\ref{construction:N,M}. Then $K_i(C^*(\GG^{N,M}_+)) \cong G_i$ for $i=0,1$.
\end{coro}

Before we study properties of $C^*(\GG^{N,M}_+)$, we make the following observation about the graph of groups $\GG^{N,M}_+$.

\begin{remark} \label{remark:G^N,M subgraph}
    Let $N$ and $M$ be as in Construction~\ref{construction:N,M}. Regardless of the groups $G_0, G_1$ used in the construction, the directed graph $\Gamma_+$ underlying $\GG^{N,M}_+$ always has the following graph (which we will call $\Gamma_0$) as a subgraph:
    
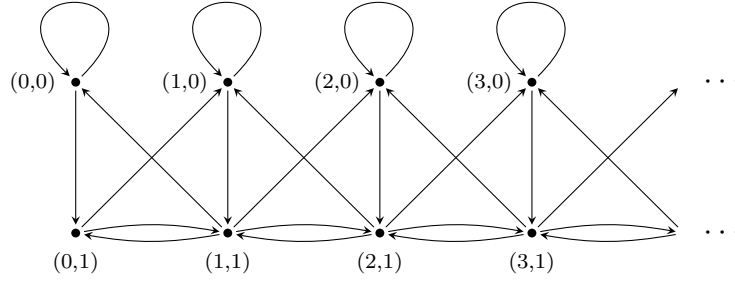
\[\begin{tikzpicture}
        \foreach \i in {0,...,3} {
            \node[vertex, label=left:{$\scriptstyle (\i,0)$}] (v\i0) at (2*\i,0) {};
            \node[vertex, label=below:{$\scriptstyle (\i,1)$}] (v\i1) at (2*\i,-2) {};
        }
        \foreach \j in {0,1} {
            \node[vertex, fill=none, label=right:{$\cdots$}] (v4\j) at (8,-2*\j) {};
        }
        
        \foreach \i in {0,...,3} {
            \pgfmathtruncatemacro\ii{\i+1}
            \draw[-stealth]
            (v\i0) edge[loop,in=135,out=45,looseness=30] (v\i0)
            (v\i0) edge (v\i1)
            
            (v\i1) edge (v\ii0)
            (v\i1) edge[bend left=10] (v\ii1)
            
            (v\ii1) edge (v\i0)
            (v\ii1) edge[bend left=10] (v\i1);
        }
    \end{tikzpicture}\]
    This is the subgraph with adjacency matrix
    \[ \begin{pmatrix} I & X \\ I & X \end{pmatrix}, \]
    where the top row of vertices corresponds to the vertices $\N \times \{0\}$, and the bottom row of vertices corresponds to the vertices $\N \times \{1\}$. Note that $\Gamma_0$ contains all edges emitted from the vertices $\N \times \{0\}$ (the top row); moreover, for each loop $e_n = ((n,0), (n,0))$, $n \in \N$ in $\Gamma_+$, the monomorphisms $\alpha_{e_n}$ and $\alpha_{\overline{e_n}}$ are defined by
    \[ \alpha_{e_n}(1_{e_n}) = 2 \times 1_{(n,0)} \quad \text{and} \quad \alpha_{\overline{e_n}}(1_{e_n}) = 3 \times 1_{(n,0)} \]
    respectively.
\end{remark}

\begin{prop} \label{prop:C*(G^N,M) Kirchberg}
    For $N$ and $M$ as in Construction~\ref{construction:N,M}, the $C^*$-algebra $C^*(\GG_+^{N,M})$ is a stable UCT Kirchberg algebra.
\end{prop}

\begin{proof}
    We apply Proposition~\ref{prop:Kirchberg conditions}. To see that $\Gamma_+$ is cofinal, note that the subgraph $\Gamma_0$ described in Remark~\ref{remark:G^N,M subgraph} is strongly connected (that is, there is a directed path from any vertex in $\Gamma_0$ to any other vertex in $\Gamma_0$): the vertices in the bottom row of the figure are strongly connected via the horizontal edges between them, and one can travel between the top and bottom rows via the vertical or diagonal edges. Since $\Gamma_0$ has the same vertex set as $\Gamma_+$, we have that $\Gamma_+$ is also strongly connected and thus cofinal.
    
    For (2) and (3), take the unique loop $e$ at $(0,0) \in \Gamma^0$. Since $n_e = 2$, the loop $e$ satisfies the condition in (2). Moreover, for $\alpha = e e \dots$, the fraction $m_{e_1}\dots m_{e_{k-1}} / n_{e_1}\dots n_{e_k}$ in (3) is equal to $3^{k-1}/2^k$, and $\limsup_{k\to\infty} \langle 3^{k-1}/2^k \rangle = \limsup_{k\to\infty} 2^k = \infty$; hence (3) is also satisfied. Then Proposition~\ref{prop:Kirchberg conditions} gives that $C^*(\GG^{N,M}_+)$ is a UCT Kirchberg algebra. Finally, since $\Gamma_+$ has infinitely many vertices, the algebra $C^*(\GG_+)$ is non-unital, and hence is stable by \cite[Theorem~1.2]{Zhang}.
\end{proof}

\begin{proof}[Proof of Theorem~\ref{thm:dgog Kirchberg}]
    Write $G_0 = K_0(A)$ and $G_1 = K_1(A)$, and let $N$ and $M$ be as in Construction~\ref{construction:N,M}. Proposition~\ref{prop:C*(G^N,M) Kirchberg} gives that $C^*(\GG^{N,M}_+)$ is a stable UCT Kirchberg algebra, and Corollary~\ref{coro:C*(G^N,M) K-theory} gives that the $K$-theory of $C^*(\GG^{N,M}_+)$ is isomorphic to the $K$-theory of $A$. Thus, the Kirchberg--Phillips classification theorem gives that $C^*(\GG_+) \cong A$.
    
    Now we show that $C_0(\partial X_{\GG_+, x}) \rtimes \pi_1(\GG, x) \cong C^*(\GG_+)$ for any $x \in \Gamma^0$. We know from Remark~\ref{rmk:stable isomorphism} that the $C^*$-algebras are stably isomorphic, so the crossed product $C_0(\partial X_{\GG_+, x}) \rtimes \pi_1(\GG, x)$ is also a UCT Kirchberg algebra with $K$-theory $(G_0, G_1)$. All that remains is to show that the crossed product is not unital (equivalently that the space $\partial X_{\GG_+, x}$ is not compact); then \cite[Theorem~1.2]{Zhang} implies that it is stable, and again the Kirchberg--Phillips classification theorem would give that $A \cong C^*(\GG_+) \cong C_0(\partial X_{\GG_+, x}) \rtimes \pi_1(\GG, x)$ as claimed.
    
    Since the spaces $\partial X_{\GG_+, x}$ are isomorphic for each $x \in \Gamma^0$, it is enough to show that $\partial X_{\GG_+, x}$ is not compact for a particular choice of $x \in \Gamma^0$. We choose $x$ to be the vertex $(0,0) \in \Gamma^0$. Consider the open cover $\{Z(\gamma G_{s(\gamma)}) : \gamma \in x F(\GG)\}$ of $\partial X_{\GG_+, x}$; we show that it does not admit a finite subcover. Indeed, let $\gamma_1, \dots, \gamma_n \in x F(\GG)$ be a finite set of $\Sigma$-normalised $\GG$-words in $x F(\GG)$ (where $\Sigma$ is some fixed set of transversals for $\GG$). Write $k$ for the maximum length of $\gamma_i$, $i=1, \dots, n$, and write $e$ for the unique loop in $\Gamma_+$ at $x$. Now consider the eventually directed ($\Sigma$-normalised) infinite $\GG$-word $\xi = \overline{e}\cdots\overline{e}1_{r(e)}ee\dots$ beginning with $k+1$ repetitions of the loop $\overline{e}$. We have that $\xi$ cannot have the form $\gamma_i \alpha$ for any $i=1, \dots, n$ and any infinite $\GG_+$-word $\alpha$, since the edge $\overline{e}$ can only appear in $\gamma_i$ and no $\gamma_i$ can contain $k+1$ copies of $\overline{e}$. Hence $\xi \notin Z(\gamma_i G_{s(\gamma_i)})$ for all $i=1, \dots, n$, so $\{Z(\gamma_i G_{s(i)})\}_{i=1}^n$ is not a cover of $\partial X_{\GG_+, x}$. This proves that $\partial X_{\GG_+, x}$ is not compact, as claimed.
\end{proof}

\bibliographystyle{plain}
\bibliography{dgog}

\begin{thebibliography}{10}

\bibitem{amenable_groupoids_book}
C.~Anantharaman-Delaroche and J.~Renault.
\newblock {\em Amenable groupoids}, volume~36 of {\em Monographies de
  L'Enseignement Math\'{e}matique [Monographs of L'Enseignement
  Math\'{e}matique]}.
\newblock L'Enseignement Math\'{e}matique, Geneva, 2000.
\newblock With a foreword by Georges Skandalis and Appendix B by E. Germain.

\bibitem{purely_infinite_AD}
Claire Anantharaman-Delaroche.
\newblock Purely infinite {$C^*$}-algebras arising from dynamical systems.
\newblock {\em Bull. Soc. Math. France}, 125(2):199--225, 1997.

\bibitem{topological_freeness}
R.~J. Archbold and J.~S. Spielberg.
\newblock Topologically free actions and ideals in discrete {$C^*$}-dynamical
  systems.
\newblock {\em Proc. Edinburgh Math. Soc. (2)}, 37(1):119--124, 1994.

\bibitem{Bass}
Hyman Bass.
\newblock Covering theory for graphs of groups.
\newblock {\em J. Pure Appl. Algebra}, 89(1-2):3--47, 1993.

\bibitem{stable_isomorphism_morita_equivalence}
Lawrence~G. Brown, Philip Green, and Marc~A. Rieffel.
\newblock Stable isomorphism and strong {M}orita equivalence of
  {$C\sp*$}-algebras.
\newblock {\em Pacific J. Math.}, 71(2):349--363, 1977.

\bibitem{BMPST}
Nathan Brownlowe, Alexander Mundey, David Pask, Jack Spielberg, and Anne
  Thomas.
\newblock {$C^*$}-algebras associated to graphs of groups.
\newblock {\em Adv. Math.}, 316:114--186, 2017.

\bibitem{BSTW}
Nathan Brownlowe, Jack Spielberg, Anne Thomas, and Victor Wu.
\newblock Group actions on multitrees and the $k$-theory of their crossed
  products, 2023.
\newblock arXiv:2311.05285.

\bibitem{CuntzKrieger}
Joachim Cuntz and Wolfgang Krieger.
\newblock A class of {$C\sp{\ast} $}-algebras and topological {M}arkov chains.
\newblock {\em Invent. Math.}, 56(3):251--268, 1980.

\bibitem{singly_aligned}
Allan Donsig, Jennifer Gensler, Hannah King, David Milan, and Ronen Wdowinski.
\newblock On zigzag maps and the path category of an inverse semigroup.
\newblock {\em Semigroup Forum}, 100(3):790--805, 2020.

\bibitem{ample_groupoids}
Carla Farsi, Alex Kumjian, David Pask, and Aidan Sims.
\newblock Ample groupoids: equivalence, homology, and {M}atui's {HK}
  conjecture.
\newblock {\em M\"{u}nster J. Math.}, 12(2):411--451, 2019.

\bibitem{crossed_products_classifiability}
Eusebio Gardella, Shirly Geffen, Julian Kranz, and Petr Naryshkin.
\newblock Classifiability of crossed products by nonamenable groups.
\newblock {\em J. Reine Angew. Math.}, 797:285--312, 2023.

\bibitem{fundamental_groupoid}
P.~J. Higgins.
\newblock The fundamental groupoid of a graph of groups.
\newblock {\em J. London Math. Soc. (2)}, 13(1):145--149, 1976.

\bibitem{purely_infinite_JR}
Paul Jolissaint and Guyan Robertson.
\newblock Simple purely infinite {$C^\ast$}-algebras and {$n$}-filling actions.
\newblock {\em J. Funct. Anal.}, 175(1):197--213, 2000.

\bibitem{Katsura}
Takeshi Katsura.
\newblock A class of {$C^*$}-algebras generalizing both graph algebras and
  homeomorphism {$C^*$}-algebras. {IV}. {P}ure infiniteness.
\newblock {\em J. Funct. Anal.}, 254(5):1161--1187, 2008.

\bibitem{Kirchberg_actions}
Takeshi Katsura.
\newblock A construction of actions on {K}irchberg algebras which induce given
  actions on their {$K$}-groups.
\newblock {\em J. Reine Angew. Math.}, 617:27--65, 2008.

\bibitem{Kirchberg}
Eberhard Kirchberg.
\newblock The classification of purely infinite {$C^*$}-algebras using
  {K}asparov’s theory.
\newblock Preprint, 1994.

\bibitem{Kumjian-Pask}
Alex Kumjian and David Pask.
\newblock {$C^*$}-algebras of directed graphs and group actions.
\newblock {\em Ergodic Theory Dynam. Systems}, 19(6):1503--1519, 1999.

\bibitem{higher_rank_graphs}
Alex Kumjian and David Pask.
\newblock Higher rank graph {$C^\ast$}-algebras.
\newblock {\em New York J. Math.}, 6:1--20, 2000.

\bibitem{purely_infinite_LS}
Marcelo Laca and Jack Spielberg.
\newblock Purely infinite {$C^*$}-algebras from boundary actions of discrete
  groups.
\newblock {\em J. Reine Angew. Math.}, 480:125--139, 1996.

\bibitem{levi_categories}
Mark~V. Lawson and Alistair~R. Wallis.
\newblock Levi categories and graphs of groups.
\newblock {\em Theory Appl. Categ.}, 32:Paper No. 23, 780--802, 2017.

\bibitem{Renault_groupoid}
Paul~S. Muhly, Jean~N. Renault, and Dana~P. Williams.
\newblock Equivalence and isomorphism for groupoid {$C^\ast$}-algebras.
\newblock {\em J. Operator Theory}, 17(1):3--22, 1987.

\bibitem{Okayasu}
Rui Okayasu.
\newblock {$C^*$}-algebras associated with the fundamental groups of graphs of
  groups.
\newblock {\em Math. Scand.}, 97(1):49--72, 2005.

\bibitem{LCSCsemigroup}
Eduard Ortega and Enrique Pardo.
\newblock The tight groupoid of the inverse semigroups of left cancellative
  small categories.
\newblock {\em Trans. Amer. Math. Soc.}, 373(7):5199--5234, 2020.

\bibitem{Paterson}
Alan L.~T. Paterson.
\newblock {\em Groupoids, inverse semigroups, and their operator algebras},
  volume 170 of {\em Progress in Mathematics}.
\newblock Birkh\"{a}user Boston, Inc., Boston, MA, 1999.

\bibitem{graph_semigroup}
Alan L.~T. Paterson.
\newblock Graph inverse semigroups, groupoids and their {$C^\ast$}-algebras.
\newblock {\em J. Operator Theory}, 48(3, suppl.):645--662, 2002.

\bibitem{Phillips}
N.~Christopher Phillips.
\newblock A classification theorem for nuclear purely infinite simple
  {$C^*$}-algebras.
\newblock {\em Doc. Math.}, 5:49--114, 2000.

\bibitem{Raeburn}
Iain Raeburn.
\newblock {\em Graph algebras}, volume 103 of {\em CBMS Regional Conference
  Series in Mathematics}.
\newblock Published for the Conference Board of the Mathematical Sciences,
  Washington, DC; by the American Mathematical Society, Providence, RI, 2005.

\bibitem{Renault_groupoid_book}
Jean Renault.
\newblock {\em A groupoid approach to {$C\sp{\ast} $}-algebras}, volume 793 of
  {\em Lecture Notes in Mathematics}.
\newblock Springer, Berlin, 1980.

\bibitem{Renault}
Jean~N. Renault.
\newblock {$C\sp{\ast} $}-algebras of groupoids and foliations.
\newblock In {\em Operator algebras and applications, {P}art 1 ({K}ingston,
  {O}nt., 1980)}, volume~38 of {\em Proc. Sympos. Pure Math.}, pages 339--350.
  Amer. Math. Soc., Providence, RI, 1982.

\bibitem{Rordam}
M.~R\o~rdam.
\newblock Classification of nuclear, simple {$C^*$}-algebras.
\newblock In {\em Classification of nuclear {$C^*$}-algebras. {E}ntropy in
  operator algebras}, volume 126 of {\em Encyclopaedia Math. Sci.}, pages
  1--145. Springer, Berlin, 2002.

\bibitem{Serre}
Jean-Pierre Serre.
\newblock {\em Trees}.
\newblock Springer-Verlag, Berlin--Heidelberg--New York, 1980.

\bibitem{Spielberg_hybrid_graphs}
Jack Spielberg.
\newblock Graph-based models for {K}irchberg algebras.
\newblock {\em J. Operator Theory}, 57(2):347--374, 2007.

\bibitem{LCSC}
Jack Spielberg.
\newblock Groupoids and {$C^*$}-algebras for left cancellative small
  categories.
\newblock {\em Indiana Univ. Math. J.}, 69(5):1579--1626, 2020.

\bibitem{Tu_UCT}
Jean-Louis Tu.
\newblock La conjecture de {B}aum-{C}onnes pour les feuilletages moyennables.
\newblock {\em $K$-Theory}, 17(3):215--264, 1999.

\bibitem{Zhang}
Shuang Zhang.
\newblock Certain {$C^\ast$}-algebras with real rank zero and their corona and
  multiplier algebras. {I}.
\newblock {\em Pacific J. Math.}, 155(1):169--197, 1992.

\end{thebibliography}

\end{document}